\documentclass[12pt]{amsart}
\usepackage{geometry} % see geometry.pdf on how to lay out the page. There's lots.
\usepackage[latin1]{inputenc}
\usepackage{amssymb}
\usepackage{graphicx}
\usepackage[all,cmtip]{xy}
\input xy
\xyoption{all}
\newcommand{\SC}{{\mathcal{C}}}

\newcommand{\SF}{{\mathcal{F}}}

\newcommand{\SJ}{{\mathcal{J}}}

\newcommand{\SL}{{\mathcal{L}}}

\newcommand{\SP}{{\mathcal{P}}}
\newcommand{\SQ}{{\mathcal{Q}}}

\newcommand{\SpS}{{\mathcal{S}}}

\newcommand{\Z}{\mathbb{Z}}
\newcommand{\C}{\mathbb{C}}

\newcommand{\N}{\mathbb{N}}
\newcommand{\R}{\mathbb{R}}

\newcommand{\RP}{\mathbb{RP}}

\newtheorem{proposition}{Proposition}[section]
\newtheorem{theorem}[proposition]{Theorem}
\newtheorem{definition}[proposition]{Definition}
\newtheorem{lemma}[proposition]{Lemma}

\newtheorem{corollary}[proposition]{Corollary}
\newtheorem{remark}[proposition]{Remark}

\title{Geometric Quantization of real polarizations via sheaves}
\author{Eva Miranda}\address{ Eva Miranda,
Departament de Matem\`{a}tica Aplicada I, Universitat Polit\`{e}cnica de Catalunya, EPSEB, Avinguda del Doctor Mara\~{n}\'{o}n, 44-50, 08028, Barcelona, Spain, \it{e-mail: eva.miranda@upc.edu}}
\author{Francisco Presas}\address{ Francisco Presas, ICMAT, UAM-UCM-UC3M, Nicol\'{a}s Cabrera, n 13-15, Campus Cantoblanco UAM, 28049 Madrid \it{e-mail: fpresas@icmat.es}}

 \thanks{Eva Miranda has been partially supported by the DGICYT/FEDER project MTM2009-07594: Estructuras Geometricas: Deformaciones, Singularidades y Geometria Integral until December 2012. Her research is partially supported by the project GEOMETRIA ALGEBRAICA, SIMPLECTICA, ARITMETICA Y APLICACIONES with reference: MTM2012-38122-C03-01  starting in January 2013. Francisco Presas has has been partially supported by the DGICYT/FEDER project MTM2010-19389. Both authors have been partially supported by ESF network CAST, \emph{Contact and Symplectic Topology}. }
\date{} % delete this line to display the current date

\begin{document}

\begin{abstract}
In this article we develop tools to compute the Geometric Quantization of a symplectic manifold with respect to a regular Lagrangian foliation via sheaf cohomology and obtain important new applications in the case of real polarizations.  The starting point is the definition of representation spaces due to Kostant.
Besides the classical examples of Gelfand-Cetlin systems due to Guillemin and Sternberg \cite{guilleminandsternberg} very few examples  of explicit computations of real polarizations are known. The computation of Geometric Quantization in \cite{guilleminandsternberg} is based on a theorem due to \'Sniatycki for fibrations \cite{sniatpaper}
which identifies the representation space with the set of Bohr-Sommerfeld leaves determined by the integral action coordinates.

 In this article we check that the associated sheaf cohomology apparatus of Geometric Quantization satisfies Mayer-Vietoris and K\"unneth formulae. As a consequence,  a new short proof of this classical result for fibrations due to \'Sniatycki is obtained. We also compute Geometric Quantization with respect to  any {\it generic} regular Lagrangian foliation on a $2$-torus and the case of the irrational flow. In the way, we recover  some classical results in the computation of foliated cohomology of these polarizations.
\end{abstract}

\maketitle

\tableofcontents

\section{Introduction}

Geometric Quantization attempts to create  a
\lq\lq (one-way) dictionary" that \lq\lq translates"  classical
systems into quantum systems.

In this way, a quantum system is
associated to a classical system in which observables (smooth
functions) become operators of a Hilbert space and the classical
Poisson bracket becomes the commutator of operators.  We refer to the classical references of Kirillov \cite{kirillov} and Woodhouse \cite{woodhouse} for an introduction to geometric quantization.

Grosso modo, following Dirac's ideas a quantization is a linear mapping of the
Poisson algebra associated to functions on the manifold into the set of operators on some
(pre-)Hilbert space, having some additional properties.

The first step in the Geometric Quantization scheme, is the prequantization which almost realizes the quantization scheme up to a condition (where the technical condition of \lq\lq completeness" of commuting operators is not required).

The guinea pig for any model of Geometric Quantization is the cotangent bundle of a manifold endowed with the canonical Liouville form.
The model for prequantization  of a cotangent bundle  was constructed by  Segal \cite{segal} in the 60's. In trying to extend this construction to general symplectic manifolds the integrality of the cohomology class of the symplectic form is required.

Unlike the quantization by deformation approach in which the classical systems can be seen as a limit of the deformation, in the geometric quantization approach, the process only has partial memory of the information of the initial system.
 This is why the choice of additional geometric structures
(polarizations) plays an important role.
The classical motivation for considering this polarization comes from the distinction between \lq\lq momentum" and \lq\lq position" which is native from mathematical formulation  of mechanical systems in the cotangent bundle of a manifold.

A desired property is
that the representation space does not depend on the polarization and this property is satisfied in the case of considering K\"{a}hler polarizations. Much attention has been devoted to the case of Geometric Quantization via Kähler quantizations but the case of real polarizations has been less explored. In this article we are interested in considering Lagrangian foliations as real polarizations.

Our point of view in this big endeavour is to
construct a \lq \lq representation space" in the case of real polarizations.

The initial idea of quantization is to consider as initial representation space the vector space of flat sections of a line bundle associated to the symplectic manifold in the directions of the polarization. These flat sections exist only locally for a generic leaf of a real polarization, making the quantization vector space usually trivial. This is why Kostant's approach to quantization \cite{kostant} proposes to replace the sheaf of flat sections along a polarization by
cohomology groups associated to this sheaf as representation spaces.
In this  article we will not discuss either the (pre)Hilbert
structure of this space nor the quantization rules.

An interesting related issue is the notion of Bohr-Sommerfeld leaves and its contribution to Geometric Quantization. A leaf of the polarization is Bohr-Sommerfeld if there are non-trivial flat sections of the bundle globally defined along it.
  The characterization of
Bohr-Sommerfeld leaves for real polarizations given by regular fibrations
is a well-known result by Guillemin and Sternberg
(\cite{guilleminandsternberg}).
In particular, the set of
Bohr-Sommerfeld leaves is discrete and is given by \lq\lq action"
coordinates that are well-defined in a neighbourhood of a compact leaf.
\begin{theorem}[Guillemin-Sternberg] If the polarization is a
regular Lagrangian fibration with compact leaves over a simply connected base
$B$, then the Bohr-Sommerfeld set is discrete and assuming that the
zero-fiber is a Bohr-Sommerfeld leaf, the Bohr-Sommerfeld set is
given by $BS=\{p\in M, (f_1(p),\dots, f_n(p))\in\mathbb Z^n\}$
where $f_1,\dots,f_n$ are global action coordinates on $B$.
\end{theorem}
This result connects with Arnold-Liouville-Mineur theorem \cite{arnold}, \cite{duistermaat} for
action-angle coordinates of integrable systems to Geometric Quantization. When we consider
toric manifolds, the base $B$ may be identified with the image of the
moment map by the toric action (Delzant polytope).

 In this case quantization is given by precisely the
following theorem of \'{S}niatycki \cite{sniatpaper}:
\begin{theorem}[\'{S}niatycki] If the leaf space $B^n$ is a
Hausdorff manifold and the natural projection $\pi:M^{2n}\to B^n$ is
a fibration with compact fibres, then all the cohomology groups
vanish except for degree half of the dimension of the manifold.
Furthermore, $\mathcal{Q}(M^{2n})=H^n(M^{2n},\mathcal{J})$, and the
dimension of $H^n(M^{2n},\mathcal{J})$ is the number of
Bohr-Sommerfeld leaves.
\end{theorem}

As a nice application of these results, Guillemin and Sternberg obtained in \cite{guilleminandsternberg} the computation of the representation space in the case of Gelfand-Cetlin systems on $\mathfrak{u(n)}^*$. Furthermore, the construction of explicit action coordinates  in \cite{guilleminandsternberg} for these systems connects directly with the field of representation theory since via the identification explained above it is possible to  compute the dimension of the spaces on which
a representation of a prescribed maximal weight is given.

Besides the case of Gelfand-Cetlin systems and some singular fibrations cases considered in \cite{mhthesis}, \cite{hamiltonmiranda} and \cite{solha} very few examples of real polarizations are known for which explicit computations of Geometric Quantization are given.  In this article we  want to make a step forward in this direction and besides presenting more examples which are missing in the literature we give a systematic method for computing Geometric Quantization of  real regular polarizations. Furthermore, besides revisiting old computations from a new perspective we explicitly compute generically the $2$-dimensional regular case.
In the way, we recover some results of foliated cohomology.

In this article we will present a \emph{Geometric Quantization computation kit} to deal with this sheaf cohomology using a fine resolution of it. As a first application we give a short proof of this theorem of \'Sniatycki stated above using a Künneth formula.

 We also find applications to compute Geometric Quantization of general Lagrangian foliations which are not in general a fibration.
In particular, we will completely determine Geometric Quantization in the case of line fields on a torus via the classification of these foliations up to topological conjugation due to Denjoy and Kneser \cite{denjoy, kneser}.

A different approach to compute Geometric Quantization  is the one of \v{C}ech cohomology. This approach was used in \cite{mhthesis}
and \cite{hamiltonmiranda} and turns out
out to be useful when
 we consider integrable systems with singularities.
We will apply our Geometric Quantization computation kit to reprove some results which were obtained via the \v{C}ech  approach
like the theorem of \'{S}niaticky \cite{sniatpaper}.

\vspace{5mm}
{\bf{Organization of this paper:}}

\vspace{5mm}

In Section 2 we present the Kostant complex  in this approach to Geometric Quantization. Even if this article focuses on the Geometric Quantization case, when it comes to the scene we will also mention the corresponding results for foliated cohomology, considered as a limit case with $\omega \to 0$. In Section 3 we prove Mayer-Vietoris and a Künneth formula for this complex. These results connect directly with the work of Bertelson \cite{bertelsonkunneth} to prove a K\"{u}nneth formula for foliated cohomology.

In Section 4 we present a short proof of the theorem of \'Sniatycki to compute Geometric Quantization of a regular fibration. In Section 5 we apply these tools to give the representation space when the polarization considered is the linear irrational flow on the torus.  In particular, we recover former results of
 El Kacimi-Alaoui for foliated cohomology \cite{elkacimi}.

 In the last section, we use the classification of foliations on the $2$-torus up to topological equivalence due to Denjoy and Kneser \cite{denjoy}, \cite{kneser} together with the Mayer-Vietoris argument to compute the Geometric Quantization associated to any generic foliation of the $2$-torus.

\section{Geometric Quantization  \`a la Kostant}
Let $(M^{2n}, \omega)$ be a closed symplectic manifold of integer class and let $L$ be an associated prequantizable line bundle, i.e. an Hermitian complex line bundle\footnote{Given a symplectic form $\omega\in H^2(M,\mathbb R)$  with integer class $[\omega]$, the lift to $H^2(M,\mathbb Z)$ is not unique in general, we have a exact sequence $Tor(H^2(M,\mathbb Z))\longrightarrow H^2(M,\mathbb Z)\longrightarrow H^2(M,\mathbb R)$, for manifolds with non-vanishing  $Tor(H^2(M,\mathbb Z))$ this line bundle is not unique.} equipped with a compatible connection whose curvature equals $-i\omega$.

A real polarization $\SP$ over $(M^{2n}, \omega)$ is a foliation whose leaves are Lagrangian with respect to $\omega$. Given a hermitian vector bundle over $(M^{2n}, \omega)$ endowed with a connection $\nabla$, we define a section $s:M \to L$ to be flat along the polarization $\SP$ if the operator
\begin{eqnarray*}
\nabla_{\SP} s: \SP  \hookrightarrow TM & \to & L \\
v & \to & \nabla_v s
\end{eqnarray*}
is the null operator. In other words, if the covariant derivative of the section vanishes along the foliation directions.

One could consider the set of flat sections of this bundle to start constructing a representation space for Geometric Quantization. In order to induce a metric on the space of sections which can endow this space with a (pre)Hilbert structure, we consider another line bundle, the so-called metaplectic correction. This line bundle also has a physical meaning since it captures some \emph{Bohr-Sommerfeld leaves} in the Kähler case which are not captured otherwise \cite{kirwin}.

In this article
we will consider the metaplectic correction (which can be seen as a correction in the line bundle) even if we do not consider here the endeavour of completing the representation space to endow it with a Hilbert space structure. This effort can be seen as a first step in this direction.

%This is why we first we need to correct the initial line bundle using a metaplectic correction and the set of flat sections that we will work with will be sections of a line bundle of type $L\otimes L'$. This is why we need first to define $L'$ and also we need to ensure that this new bundle admits a flat connection so that the associated problem of flat sections has a solution. The natural candidate for $L'$ is the square root of the complex determinant bundle. Additional conditions on the original line bundle have to be imposed in order for this square root to  exist.

We say that $\SP$ admits a metalinear correction if the complex determinant bundle $N= \bigwedge^n_{\C} (\SP \otimes_{\R} \C)$ admits a square root $N^{1/2}$. This is equivalent to  the evenness of the first Chern class of $N$,  $c_1(N)$. This can be easily proven using the isomorphism between  Picard group of isomorphism classes of complex line bundles and $H^2(M,\mathbb Z)$. The group structure on the Picard group is an  abelian group with group operation the tensor product and  therefore $c_1(L\otimes L')=c_1(L)+ c_1(L')$. Thus  $\SP$  admits a metaplectic correction if and only if there exists an element $e\in H^2(M,\Z)$ such that $2e=c_1(N)$.

%The existence of this bundle $N^{1/2}$ is a necessary condition in the foliation to apply what follows. We will check it in different examples.

Let us prove that $N$ admits a natural flat connection. This will be needed to guarantee the existence of solutions of the associated equation of flat sections.

  Weinstein \cite{weinstein} defined a natural  flat connection along a Lagrangian foliation using symplectic duality. This will be needed to define a canonical flat connection on $N$ .

Let us recall here Weinstein's construction:
The classical Bott connection (\cite{Bott})  can be defined as a partial connection associated to a foliation $\mathcal P$. This allows to define the covariant derivative of vector fields which belong to the normal bundle of the foliation with respect to vector fields of the distribution.

  Given $X$, a vector field of the polarization, we denote by $\overline Y$ a vector field in the normal bundle $\mathcal N$.  Let us denote by $p:TM\longrightarrow \mathcal N$ the projection on the normal bundle. The Bott connection is defined as $\nabla^B_X(\overline Y)=p([X,Y])$ where $Y$ is any vector field on $TM$ which projects to $\overline Y$.

 This connection is flat. We now use the fact that symplectic duality establishes a natural isomorphism between the tangent and cotangent bundles of $M$ to induce a flat connection on $\SP$ (given a Lagrangian leaf  of the foliation the normal bundle to the leaf can be easily identified with the cotangent bundle of the leaf). In this way the bundle $N$ acquires a canonical flat connection and so it does the bundle $N^{1/2}$. Let us call this connection $\triangle^B$.

The initial connection on the flat bundle $\nabla$ together with the new connection $\triangle^B$ can be used to define a partial connection $\nabla'$ on the bundle $L\otimes N^{1/2}$. We now consider the equation,
$$\nabla'_v s=0, v\in \SP,$$

\noindent where $s$ is a section of the bundle $L\otimes N^{1/2}$. The flatness of $\triangle^B$ and the Lagrangian character of the polarization guarantee that this equation has a local solution.

As observed by \'{S}niaticky in \cite{sniatpaper}, if the leaves have non-trivial fundamental group then there exist no global flat section along all the leaves of the polarization. Thus, it makes sense to consider the sheaf of flat sections and consider as a first candidate for Geometric Quantization the cohomology groups associated to this sheaf.
Let us formalize this idea: We denote by
$\SpS$ be the sheaf of sections of the line bundle $L\otimes N^{1/2}$; and by $\SJ$  the sheaf of flat sections of the bundle $L\otimes N^{1/2}$ along the real polarization $\SP$. Denote by $H^i(M^{2n}, \SJ)$ the $i$-th sheaf cohomology group associated to the sheaf $\SJ$. We define Geometric Quantization using these cohomology groups,
\begin{definition}
The Geometric Quantization space of $(M, \omega)$ with respect to $\SP$ is defined as,
$$ \SQ(M^{2n}, \SP)= \bigoplus_{i \in \N} H^i(M^{2n},\SJ).$$
\end{definition}

\begin{remark}
This cohomology with coefficients in the sheaf of flat sections admits computations \`{a} la \v{C}ech  and thus, a priori, $H^i(M^{2n},\SJ)=0$ for $i>2n$.  Indeed, as we will see, this sheaf admits a fine resolution and we may even apply a computation twisting the De Rham complex that computes foliated cohomology yielding $H^i(M^{2n},\SJ)=0$ for $i>n$.
\end{remark}

\vspace{5mm}

{\bf{The Kostant complex}}

\vspace{5mm}

Let  $\Omega_{\SP}^i$ denote the associated sheaves of sections of the vector bundles $\bigwedge^i \SP$, i.e.
$$\Omega_{\SP}^i(U) = \Gamma(U, \bigwedge^i \SP). $$ Let $\SC$ be the sheaf of complex-valued functions that are locally constant along $\SP$. Then, we consider the natural (fine) resolution
$$ 0 \to \SC \stackrel{i}{\to} \Omega_{\SP}^0 \stackrel{d_{\SP}}{\to}  \Omega_{\SP}^1 \stackrel{d_{\SP}}{\to} \Omega_{\SP}^1 \stackrel{d_{\SP}}{\to} \Omega_{\SP}^2 \stackrel{d_{\SP}}{\to} \cdots ,$$
The differential operator $d_{\SP}$ is the restriction of the exterior differential along the  directions of the distribution. This is the standard resolution used to compute the foliated (or tangential) cohomology of the foliation $\SP$ (see for instance \cite{elkacimi}).

We can now use this fine resolution of $\SC$ to construct
 a fine resolution of the sheaf $\SJ$. We just need to \lq\lq twist" the previous resolution with the sheaf $\SJ$. It produces the following exact sequence,
$$0 \to  \SJ \stackrel{i}{\to} \SpS \stackrel{\nabla'_{\SP}}{\to} \SpS \otimes \Omega^1_\SP  \stackrel{\nabla'_{\SP}}{\to} \SpS \otimes \Omega^2_\SP \to \cdots $$

This is called the Kostant complex. The cohomology of this complex computes exactly $H^i(M, \SP)$ and therefore computes Geometric Quantization.

In the next sections, we will develop some tools and techniques to compute the cohomology of this complex.

\section{Geometric Quantization computation kit}\label{sec:functorial}
In this Section we provide some algebraic tools to compute the Geometric Quantization associated to a real polarization.

\subsection{A Mayer-Vietoris sequence}
The classical Mayer-Vietoris theorem, proved for singular cohomology, also works for the cohomology of any sheaf over quite general topological spaces (see \cite{Br97}, pg. 94). For the sake of completeness we are going to provide a proof for our sheaf $\SJ$.
%
%We will start proving a particular case that it is the one that will be used in all the examples.

Take $U$ and $V$ a pair of open sets covering $M$. Then there is a sequence of inclusions
$$ M \leftarrow U \sqcup V \leftleftarrows U \cap V.$$

This induces a sequence as follows
$$\xymatrix{\SpS \otimes \Omega_{\SP}^*(M)\ar[r]^-r & (\SpS \otimes \Omega_{\SP}^*(U)) \oplus  (\SpS \otimes \Omega_{\SP}^*(V))
\ar@<+.7ex>[r]^-{r_0}\ar@<-.7ex>[r]_-{r_1} & \SpS\otimes \Omega_{\SP}^*(U \cap V)}$$
%r_0 \\ \to \\ \to \\ r_1  \end{array} \right. \SpS\otimes \Omega_{\SP}^*(U \cap V),$$
\noindent given by the restriction map $r, r_0$ and $r_1$. Substracting the two morphisms on the right side we get the following sequence,
\begin{equation}
\xymatrix{0 \ar[r] & \SpS \otimes \Omega_{\SP}^*(M) \ar[r]^-{r} &( \SpS \otimes \Omega_{\SP}^*(U)) \oplus ( \SpS \otimes \Omega_{\SP}^*(V)) \ar[r]^-{r_0-r_1}& \SpS \otimes \Omega_{\SP}^*(U \cap V) \ar[r] & 0} \label{eq:mayer}
\end{equation}
We have the following,
\begin{theorem}
The sequence \eqref{eq:mayer} is exact.
\end{theorem}
\begin{proof}
The injectivity of the morphism $r$ is obvious, and so it is the exactness at $(\SpS \otimes\Omega_{\SP}^*(U))\oplus  (\SpS \otimes \Omega_{\SP}^*(V))$. To check the surjectivity of $r_0-r_1$ we take a partition of the unity relative to the covering $U,V$ given by functions $\chi_U$ and $\chi_V$. Now for any $\alpha \in \SpS \otimes \Omega_{\SP}^*(U \cap V)$, we define
\begin{eqnarray}
\beta_U & = & \, \, \, \chi_V \cdot \alpha \in  \SpS \otimes \Omega_{\SP}^*(U), \label{eq:extU} \\
\beta_V & = & -\chi_U \cdot \alpha \in \SpS \otimes \Omega_{\SP}^*(V). \label{eq:extV}
\end{eqnarray}
We immediately obtain that $r_0(\beta_U)-r_1(\beta_V)= \alpha$.
\end{proof}

So after playing this game at all degrees,  we obtain the following commutative exact diagram,
\vspace{5mm}
$$
\xymatrix{
& 0 \ar[d] & 0 \ar[d] & 0 \ar[d]  & \\
0 \ar[r] & \SpS \otimes \Omega_{\SP}^0(M) \ar[r]^-r \ar[d]_{\nabla_\SP} & (\SpS \otimes \Omega_{\SP}^0(U))\oplus  (\SpS \otimes \Omega_{\SP}^0(V) )\ar[r]^-{r_0-r_1} \ar[d]_{\nabla_\SP} &  \SpS \otimes \Omega_{\SP}^0(U \cap V) \ar[r] 0 \ar[d]_{\nabla_\SP} & 0 \\
0 \ar[r] & \SpS \otimes \Omega_{\SP}^1(M) \ar[r]^-r \ar[d]_{\nabla_\SP} & (\SpS \otimes \Omega_{\SP}^1(U))\oplus  (\SpS \otimes \Omega_{\SP}^1(V) )\ar[r]^-{r_0-r_1} \ar[d]_{\nabla_\SP} &  \SpS \otimes \Omega_{\SP}^1(U \cap V) \ar[r] 0 \ar[d]_{\nabla_\SP} & 0 \\
& \vdots & \vdots & \vdots
}
$$
\vspace{5mm}

Observe that vertically we have the Kostant complex that computes cohomology with coefficients in a sheaf.
We apply the snake's Lemma \cite{atiyah} to prove the following,

\begin{corollary} \label{coro:Mayer}
The following sequence is exact \newline
\begin{equation}
\xymatrix{
0 \ar[r] & H^0(M, \SJ) \ar[r] & H^0(U, \SJ)\oplus H^0(V, \SJ) \ar[r] & H^0(U \cap V, \SJ) \ar[lld]^{\delta} \\
& H^1(M, \SJ) \ar[r] & H^1(U, \SJ)\oplus H^1(V, \SJ) \ar[r] & H^1(U \cap V, \SJ) \ar[lld]^{\delta}   \\
& H^2(M,\SJ) }
\end{equation} \label{eq:MV}
\end{corollary}
We recall how to define the connecting operator $\delta$ because it will be needed later. We take a class $a \in H^i(U \cap V, \SJ)$, represented by a form $\alpha \in \SpS \otimes \Omega_{\SP}^i(U \cap V)$. From the previous diagram we know that there are forms  $(\beta_U, \beta_V) \in (\SpS \otimes \Omega_{\SP}^i(U))\oplus  (\SpS \otimes \Omega_{\SP}^i(V))$, defined via the formulae \eqref{eq:extU} and \eqref{eq:extV}. Then we take the image under the vertical morphism to obtain $(\nabla_\SP \beta_U, \nabla_\SP \beta_V) \in (\SpS \otimes \Omega_{\SP}^{i+1}(U))\oplus  (\SpS \otimes \Omega_{\SP}^{i+1}(V))$. Since $a \in H^i(U \cap V, \SJ)$,  commutativity of the diagram  yields $(r_0-r_1)(\nabla_\SP \beta_U, \nabla_\SP \beta_V)=0$. Hence there is a form $\gamma \in \SpS \otimes \Omega_{\SP}^{i+1}(M)$ such that $r(\gamma)= (\nabla_\SP \beta_U, \nabla_\SP \beta_V)$. Again the commutativity of the diagram yields $\nabla_\SP (\gamma)=0$ and thus it defines a closed form whose class in $H^{i+1}(M, \SJ)$  satisfies by construction $\delta([a])=[\gamma]$ (and this is how the connecting operator is defined). It is easy to check that this definition does not depend on the form representing the class $[a]$.

\subsection{Mayer-Vietoris for closed domains} \label{subsec:closed} The space of sections of the sheaf $\SJ$ can be defined for closed sets $C \subset M$. Take the directed system of sections $\SJ(A_i)$ of open sets $A_i$ containing the closed set $C$. It is directed by the restriction morphism. Then, we define the space of sections over $C$ as
$$ \SJ(C)= \lim_{\longrightarrow} A_i,$$
the direct limit of the system. The proof of the previous Mayer--Vietoris result can be easily adapted to the case of closed sets (see \cite{BT82} for all the details). It is stated as follows.
Take $C_0$ and $C_1$ a pair of closed sets covering $M$. Then there exists, as in the previous Subsection, a sequence of inclusions
$$ M \leftarrow C_0 \sqcup C_1 \leftleftarrows C_0 \cap C_1.$$
Again, we obtain
$$\xymatrix{\SpS \otimes \Omega_{\SP}^*(M)\ar[r]^-r & (\SpS \otimes \Omega_{\SP}^*(C_0)) \oplus  (\SpS \otimes \Omega_{\SP}^*(C_1))
\ar@<+.7ex>[r]^-{r_0}\ar@<-.7ex>[r]_-{r_1} & \SpS\otimes \Omega_{\SP}^*(C_0 \cap C_1)}$$
%r_0 \\ \to \\ \to \\ r_1  \end{array} \right. \SpS\otimes \Omega_{\SP}^*(U \cap V),$$
\noindent given by the restriction map $r, r_0$ and $r_1$. Substracting the two morphisms on the right side we get the following sequence,
\begin{equation}
\xymatrix{0 \ar[r] & \SpS \otimes \Omega_{\SP}^*(M) \ar[r]^-{r} &( \SpS \otimes \Omega_{\SP}^*(C_0)) \oplus ( \SpS \otimes \Omega_{\SP}^*(C_1)) \ar[r]^-{r_0-r_1}& \SpS \otimes \Omega_{\SP}^*(C_0 \cap C_1) \ar[r] & 0} \label{eq:mayer2}
\end{equation}
We have the following,
\begin{theorem}
The sequence \eqref{eq:mayer2} is exact and thus it induces a long exact sequence in cohomology.
\end{theorem}

\subsection{A K\"unneth formula}
The classical K\"unneth formula also holds with great generality for the cohomology of a sheaf (\cite{Br97}). It works for the Geometric Quantization scheme in a generalized form. Let $(M_1, \SP_1)$ and $(M_2, \SP_2)$ be a pair of symplectic manifolds endowed with Lagrangian foliations. The natural cartesian product for the foliations is Lagrangian with respect to the product symplectic structure. The induced sheaf of flat sections associated to the product foliation will be denoted $\SJ_{12}$. Note that we use the pre-quantum bundle and metaplectic corrections defined as pull-backs and tensor products of the ones defined over $M_1$ and $M_2$

There is a natural morphism,
\begin{equation}
\Psi: H^*(M_1, \SJ_1) \otimes H^*(M_2, \SJ_2) \to H^*(M_1 \times M_2, \SJ_{12})  \label{eq:Ku}
\end{equation}
induced by pull-back of the classes through the natural projections. Furthermore we can prove the following K\"{u}nneth formulae,
\begin{theorem}[K\"{u}nneth formula for Geometric Quantization] \label{thm:Kun}
There is an isomorphism
$$ H^n(M_1 \times M_2,  \SJ_{12}) \cong \bigoplus_{p+q=n} H^p(M_1, \SJ_1) \otimes H^q(M_2, \SJ_2), $$
\noindent whenever the Geometric Quantization associated to $(M_1,\SJ_1)$ has finite dimension, $M_1$ is compact  and $M_2$ admits a \emph{good} covering.

\end{theorem}
The isomorphism is defined by the inverse of the map $\Psi$ above and the core of the proof will be to show its existence.

\begin{remark} In this article the condition of  good covering will be the one as stated  in \cite{Br97}. This condition is automatically fulfilled in particular if $M_2$ has finite topology or it is compact. The compactness condition on $M_1$ can indeed  be relaxed to the following one:  $M_1$ is a submanifold of a compact manifold. What we will do is to establish the K\"{u}nneth formula for closed balls and then use compactness.

%\footnote{This requires to use the notion of Geometric Quantization of manifolds with boundary. For classical references see \cite{gilkey} and \cite{tianzhang}.}.
\end{remark}
\begin{remark}

In \cite{Br97} a more general formula is given, where a torsion of a complex is present. Whenever the complex is torsionless, we obtain a clean K\"unneth formula. One way to look at this result is as a torsionless case of a generalized K\"unneth formula \`{a} la Grothendieck \cite{grothendieck}.

Many authors have studied more general conditions under which a K\"unneth formula holds for sheaf cohomology with infinite dimension. See for instance the works of Kaup \cite{kaup} and Grothendieck \cite{grothendieck}  where a K\"unneth formula is given in terms of completion of the tensor product:

$$ H^n(M_1 \times M_2,  \SJ_{12}) = \bigoplus_{p+q=n} H^p(M_1, \SJ_1) \hat{\otimes} H^q(M_2, \SJ_2). $$

This approach using nuclear spaces has also been adopted by Bertelson in \cite{bertelsonkunneth} for foliated cohomology. In particular, Bertelson obtains a similar result to our K\"{u}nneth formula  for foliated cohomology under similar hypothesis (imposing already compactness of one of the factor manifolds).  This foliated cohomology result can be reinterpreted as the \lq\lq zero limit\rq\rq  case of Geometric Quantization.
\end{remark}

The proof follows very closely the de Rham cohomology case. We start by proving the following proposition which automatically implies the K\"{u}nneth formula in  case one of the factors is an open set with the property that  the leaves restricted to the domain $U$ are contractible and the leaf space is also contractible.
We will call these open sets \emph{cotangent balls}. Its compactification, whenever the leaves and the leaf space are still contractible, will be called a \emph{closed cotangent ball}.\footnote{ Notice that for this we need to extend the definitions of  Geometric Quantization to that of  manifolds with boundary. For classical references see \cite{gilkey} and \cite{tianzhang}. However, it follows the definitions provided in Subsection \ref{subsec:closed}.}

\begin{proposition}  \label{prop:coba}
Let $U$ be a closed cotangent ball and $M$ a compact manifold with finite dimensional Geometric Quantization. The following equality holds,
$$ H^n(M \times U,  \SJ_{10}) = H^n(M, \SJ) \otimes H^0(U, \SJ_0), $$
\end{proposition}

%\begin{remark}
%We will check in the proof that the cotangent balls satisfy $H^i(U, \SJ_0)=0$ for all $i>0$ and so this Lemma indeed proves the K\"unneth formula for the particular case in which one of the factors is a cotangent ball.
%
%\end{remark}

Before proving this proposition, we need the following lemma,
\begin{lemma}\label{lem:decomposing}
Given an element $\alpha \in (\SpS \otimes \Omega_{\SP}^p)(M\times U)$ which is closed by the Kostant differential $\nabla_{\SP}$, we can always find another element $\alpha'$ in the same cohomology class in the Kostant complex such that $\alpha'\in \SpS(M \times U) \otimes \Omega_{\SP}^{p}(M)$.
\end{lemma}

\begin{proof}
Assume that we are in the case $M \times U$, for a cotangent ball $U \subset \R^2$. The proof can be adapted for higher dimensions of $U$ just by sophisticating the notation. Denote $\SP_M$ the polarization in $M$ and $(L_M, \nabla_M)$ the quantization bundle over $M$. In the same way, the cotangent ball has as quantization bundle the restriction of the quantization bundle on $\R^2$ which we denote by $L_U$ with connection $\nabla_U$. The product connection will be denoted by $\nabla = \nabla_M \boxplus \nabla_U$.

Let us take a system of coordinates over $U$ in which we can trivialize the prequantization bundle in the horizontal directions starting with a parallel section along the vertical axis\footnote{This has been known in the literature of Geometric Quantization as the existence of a \lq\lq trivializing section" for neighbourhoods in the case of regular foliations.}. With respect to this trivialization the connection on the quantization bundle $L_U$ is written,
$$ \nabla_{U}  = d - i \phi(x,y) dy. $$

Now, take an element  $\alpha\in \SpS \otimes \Omega_{\SP}^p(M\times U)$. It can be written as,
\begin{equation}
\alpha= \hat{\alpha} + \beta \wedge dy, \label{eq:forma}
\end{equation}
where $\hat{\alpha} \in  \SpS \otimes \Omega_{\SP}^{p}(M)$ and $\beta \in  \SpS \otimes \Omega_{\SP}^{p-1}(M)$.

Observe that $\alpha$ is closed in the Kostant complex and therefore $\nabla_\SP \alpha=0$. Recall that $U$ is a cotangent ball, therefore
$$ U=  \{ (x, y) \in \R^2 : - \epsilon_1 \leq x \leq \epsilon_2, h_1(x) \leq y \leq h_2(x) \}. $$
Take $h_3(x)= \frac{h_1(x)+h_2(x)}{2}$. Define $\gamma \in \SpS \otimes \Omega_{\SP}^{p-1}(M)$ satisfying the differential equation
$$ \frac{\partial \gamma}{\partial y} - i\phi(x,y) \gamma = \beta, $$
with the initial conditions $\gamma(p, x,h_3(x))= 0$. This is a first order linear ordinary differential equation with \lq\lq parameter" $p\in M$ and fixed initial conditions. Thus it has a unique solution.

 It is simple to check that the following equality holds $\nabla \gamma= \nabla_M \gamma + \beta \wedge dy$. Thus, replacing this expression in (\ref{eq:forma}), we obtain
$$
\alpha= (\hat{\alpha} - \nabla_M \gamma) + \nabla \gamma= \alpha' + \nabla \gamma,
$$
we conclude by observing that $\alpha' = \hat{\alpha} - \nabla_M \gamma \in \SpS(M \times U) \otimes \Omega_{\SP}^{p}(M)$ and $\nabla \gamma$ is obviously exact.
\end{proof}

\begin{figure}[ht]
\includegraphics[scale=0.3]{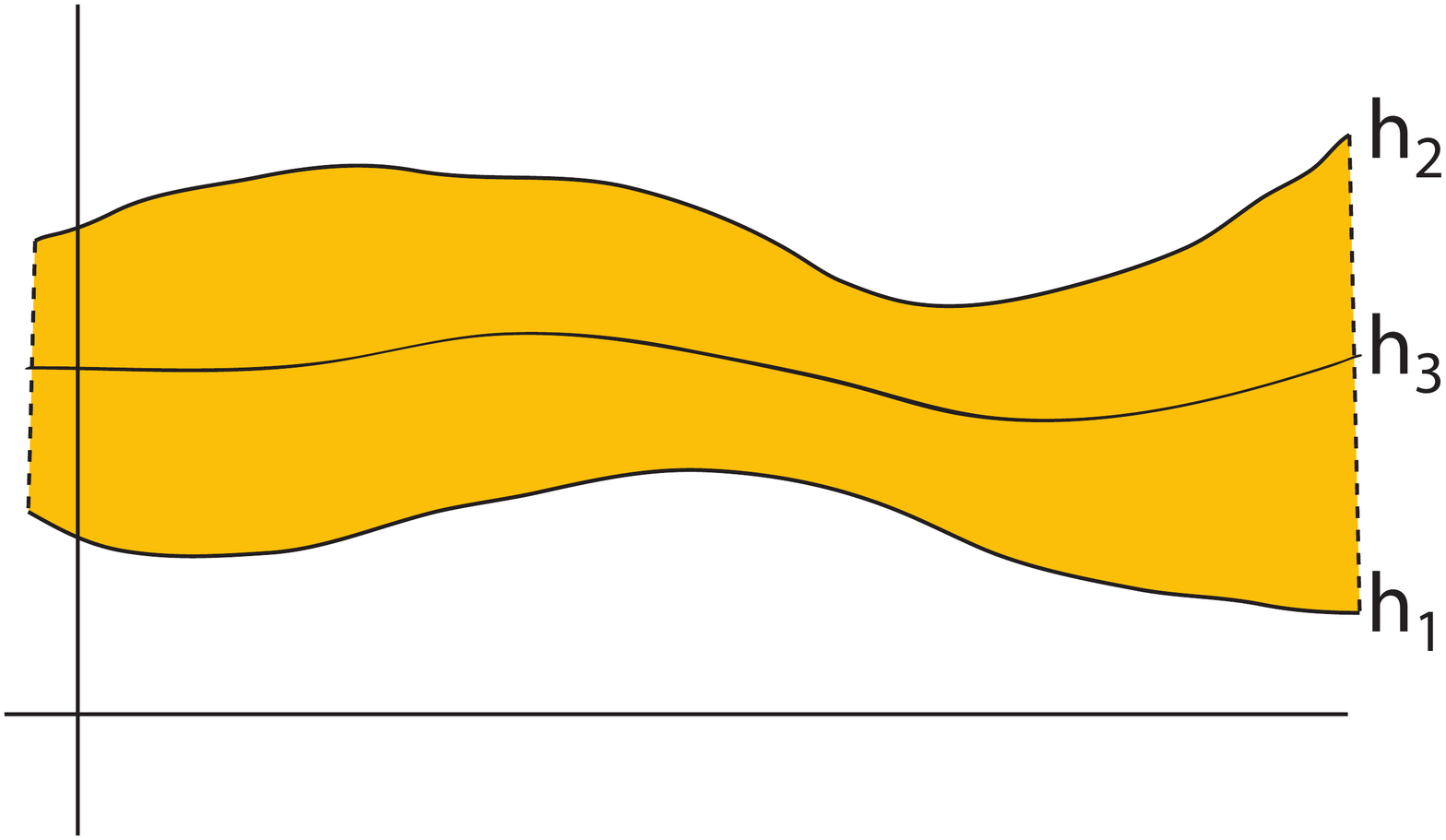}
\label{fig:h_3}
\end{figure}

 We now recall a few facts about orthogonal series of functions in interpolation theory. Let $f: S^1 \to \C$ be a continuous function over the circle. Denote by $\hat{f}: \Z \to \C$ its Fourier coefficients. We have the following relation between the decay of the coefficients and the smoothness of the function.
\begin{theorem}[Theorem 3.2.9 and Proposition 3.2.12 in \cite{Gr09}] \label{thm:Graf}
For any $s \in \Z^+$, the following two statements hold:
\begin{enumerate}
\item If $f$ is a function of $C^s$-class then we have $\lim_{k \to \pm \infty}\frac{|\hat{f}(k)|}{1+ |k|^s}=0$,
\item If $\lim_{k \to \pm \infty}\frac{|\hat{f}(k)|}{1+ |k|^{s+1}}=0$ then we have  $f$ is a function of $C^s$-class.
\end{enumerate}
\end{theorem}

\noindent {\emph{Proof of Proposition \ref{prop:coba}.}}

The parametrized Poincar\'e Lemma immediately implies that for the cotangent balls we have $H^i(U,\SJ_0)=0$ for $i>0$ and $H^0(U, \SJ_0)=C^{\infty}(\R^n,\C)$. Assume that we are in the product case $M \times U$, for a cotangent ball $U \subset \R^2$. The proof can be easily adapted for higher dimensions of $U$.

Recall the natural morphism (\ref{eq:Ku}), in particular, produces  a map
$$\Delta: H^n(M, \SJ) \otimes H^0(U, \SJ_0) \to H^n(M \times U,  \SJ_{10}). $$

Let us check that this map is injective. Consider a pair of elements $(a,f)$ with $a \in \SpS (M \times U) \otimes \Omega_{\SP}^n(M)$ satisfying that $\nabla_M a=0$ and $f \in \Omega_{\SP}^0(U)=C^{\infty}[0,1]$ such that there exists an element $b\in  \Omega_{\SP}^{n-1}(M \times U)$ satisfying that $a\cdot f= \nabla b$ (the pair $(a,f)$ goes to the zero class). Since $f\neq 0$, there is a number $x_0\in[0,1]$ such that $f(x_0) \neq 0$.
 Recall that $U$ is a cotangent ball, therefore
$$ U=  \{ (x, y) \in \R^2 : - \epsilon_1 \leq x \leq \epsilon_2, h_1(x) \leq y \leq h_2(x) \}. $$
Take $h_3(x)= \frac{h_1(x)+h_2(x)}{2}$ and consider the inclusion map
\begin{eqnarray*}
i_0: M & \to & M \times U \\
p & \to & (p,x_0, h_3(x_0)).
\end{eqnarray*}
Define the element $\tilde{b}= i_0^* b$.  Since the covariant derivative commutes with the restriction, we obtain $\nabla_M \tilde{b} =i_0^* \nabla b$. This yields  $\nabla_M (\frac{1}{f(x_0)} \tilde{b}) =a$ and hence $a$ is exact and $[a]=0$. This proves  injectivity.

We will now check the surjectivity of this mapping \footnote{ In the case of foliated cohomology surjectivity of this map was not well sorted in  works prior to the work of Bertelson \cite{puta}. In our opinion a complete proof of this fact for the limit case of foliated cohomology is only achieved in \cite{bertelsonkunneth}.}. Select an element $A\in H^n(M \times U,  \SJ_{10})$. By Lemma \ref{lem:decomposing}, there is a representative $\alpha$ of the class $A$ such that $\alpha \in \SpS(M \times U) \otimes \Omega_{\SP}^{n}(M)$. Recall that  $\nabla_\SP \alpha=0$ yields in coordinates $(x,y)\in U \subset \R^2$ the following equations,
\begin{equation}
\nabla_{\frac{\partial}{\partial y}} \alpha=0. \label{eq:determined}
\end{equation}
Consider the inclusion,
\begin{eqnarray*}
i_h:M \times [-1,1] & \to & M \times U \\
(p,x) & \to & (p,x,h_3(x))
\end{eqnarray*}
 and the pull-back form obtained via $i_h$, $\tilde{\alpha}= i_h^* \alpha$. Observe that, by equation (\ref{eq:determined}), the form $\tilde{\alpha}$ completely determines $\alpha$. The form $\tilde{\alpha}$ can be understood as a $1$-parametric family $\tilde{\alpha}_t(p)=\tilde{\alpha}(p,t)$, for $t\in[-1,1]$, with $\tilde{\alpha}_t$ a closed element of $\Omega_{\SP}^{n}(M)$. Thus, there is a family ${\tilde{\alpha}'}_t$, for $t\in [-2,2]$ satisfying:
\begin{enumerate}
\item $\tilde{\alpha}'_t=\tilde{\alpha}_t$, for $t\in[-1,1]$,
\item $\tilde{\alpha}'_t=0$ for $|t|\geq 3/2$,
\item $\nabla_M \tilde{\alpha}'_t =0$.
\end{enumerate}
%Choose a real positive constant $Q\in \R^+$ satisfying that
%$$ |h_1(x)| <M, |h_2(x)|< M, \forall x\in [-1,1].$$
%It exists by compactness. Find a cut-off function $\chi: [-1,1] \times \R \to [0,1]$ with the following properties:
%\begin{enumerate}
%\item $\chi(x,y)=1$, for all $(x,y) \in U.$
%\item $\chi(x,y)=0$, whenever $|y|\geq Q$
%\end{enumerate}
%Fix the compact ball $V=[-1,1] \times [-2Q, 2Q]$. We extend, by using formula (\ref{eq:determined}), the form $\alpha \in \SpS(M \times U) \otimes \Omega_{\SP}^{n}(M)$ to a form $\alpha' \in \SpS(M \times U) \otimes \Omega_{\SP}^{n}(M)$
This defines an element $\tilde{\alpha}'\in \SpS(M \times [-2,2]) \otimes \Omega_{\SP}^{n}(M)$ and by applying the differential equation (\ref{eq:determined}), we can extend it to a unique closed form $\alpha' \in \SpS(M \times V) \otimes \Omega_{\SP}^{n}(M)$, where the closed domain $V\supset U$ satisfies $V \bigcap (\R \times \{0\})= [-2,2]$.

Fix a point $p\in M$. Define the inclusion map $e_p: [-2,2] \to M \times [-2,2]$. Now we obtain the map
\begin{eqnarray*}
\tilde{\alpha}_p= e^*_p \tilde{\alpha}: [-2,2] & \to  & \SpS(M \times [-2,2])_p \otimes \Omega_{\SP}^{n}(M)_p\simeq \SpS(M)_p \otimes \Omega_{\SP}^{n}(M) \\
t & \to & \tilde{\alpha}(p,t).
\end{eqnarray*}

Remember that $\tilde{\alpha}_p$ extends to a smooth section $\tilde{\alpha}_p: S^1=\R/ (4\Z) \to \SpS(M \times [-2,2])_p \otimes \Omega_{\SP}^{n}(M)_p$. The Fourier coefficients of $\tilde{\alpha}_p$ are computed as,
\begin{equation}
\tilde{\alpha}_{p}(m) = \int_{-2}^2 \tilde{\alpha}_p(t) \cdot e^{i \pi m t /2} dt, \label{eq:Four}
\end{equation}
that are obviously smooth in $p\in M$. This equation defines elements $\tilde{\alpha}(m) \in\SpS(M)_p \otimes \Omega_{\SP}^{n}(M)$ that are closed since the covariant derivative commutes with the integration in  formula (\ref{eq:Four}).

Since the space $M$ is compact and the variation of parameters is continuous, Theorem \ref{thm:Graf} yields the following inequality  for any $s>0$,
\begin{equation*}
|\tilde{\alpha}_{p}(m)| \leq C \cdot |m|^s,
\end{equation*}
where $C>0$ is a constant that does not depend on $p \in M$ ( but which may depend on $s$). So, it can be rewritten as,
\begin{equation}
|\tilde{\alpha}(m)|_{C^0} \leq C \cdot |m|^s, \label{eq:norm}
\end{equation}
Since the cohomology groups are defined as,
$$H^n(M, \SJ_1)=\frac{ \SpS(M) \otimes \Omega_{\SP}^{n}(M)}{ Im \nabla ( \SpS(M) \otimes \Omega_{\SP}^{n-1}(M))},$$
the topology of $H^n(M, \SJ_1)$ is defined as the vector space quotient topology, induced out of the $C^0$-norm in $\Omega_{\SP}^{n}(M)$. Therefore, by using the equation (\ref{eq:norm}) and the decrease of the norm in the projection we obtain
\begin{equation}
|[\tilde{\alpha}(m)]| \leq |\tilde{\alpha}(m)|_{C^0} \leq C \cdot |m^s|, \label{eq:class_norm}
\end{equation}
where $[\tilde{\alpha}(m)]$ is the class represented by the element $\tilde{\alpha}(m)\in \SpS(M) \otimes \Omega_{\SP}^{n}(M)$ in $H^n(M, \SJ_1)$.

Now, we can check that $H^n(M, \SJ) \otimes H^0(V, \SJ_0)= C^{\infty}([-2,2], H^n(M, \SJ))$, i.e. the smooth maps from the interval to the finite vector space $H^n(M, \SJ)$ (we are using the differential equation (\ref{eq:determined}) to uniquely extend the section from $M\times [-2,2]$ to $M \times V$).
Once this identification is done, we can prove surjectivity  of $\Delta$ in the following way:
 Define the formal series,
$$ \hat{\alpha} = \Sigma_{m\in \Z} [\tilde{\alpha}(m)] e^{i \pi m t /2}, $$
\noindent we can use again Theorem \ref{thm:Graf} in combination with the inequality (\ref{eq:class_norm}) to conclude that $\hat{\alpha}$ is smooth. i.e. $\hat{\alpha} \in C^{\infty}([-2,2], H^n(M, \SJ))$. Denote its restriction by $\hat{\alpha}'\in C^{\infty}([-1,1], H^n(M, \SJ))$. The following equality holds $\Delta \alpha = \hat{\alpha}'$ and this proves surjectivity of $\Delta$.

Note that we have worked out the proof for $2-$dimensional cotangent balls,
 for the case of  cotangent balls $U$ having dimension $2k$, we can still apply Poincar\'e Lemma to make sure that the $y$-directions in the ball are constant. In this case we  need to work with Fourier coefficients in the space of maps $C^{\infty}(\mathbb T^k,H^n(M, \SJ))$. All the arguments go through; the key point in the proof being  the finite dimensionality of the space $H^n(M,\SJ_1)$.
$\hfill \Box$

We now proceed with the proof of K\"{u}nneth
 for the general case.

\begin{proof}{\emph{of Theorem \ref{thm:Kun}}}

Denote by $\SP_0$ the vertical foliation in $\R^{2m}$, with coordinates $(x_1, \ldots x_m, y_1, \ldots, y_m)$, whose leaves are defined by the equations
$$ x_i= c_i, c_i\in \R, i=1, \ldots, m. $$
With respect to the standard symplectic structure the foliation $\SP_0$ is standard. Choose a primitive $1$-form as the connection form for the quantization line bundle $L \otimes N^{1/2}$ over $\R^{2n}$ and denote by $\SJ_0$ the associated sheaf of flat sections.

The rest of the proof is standard and follows step by step pg.49 in \cite{BT82}. We outline the main ideas. There is a natural morphism
\begin{equation}
\Psi: H^*(M_1, \SJ_1) \otimes H^*(M_2, \SJ_2) \to H^*(M_1 \times M_2, \SJ_{12})  \label{eq:Ku}
\end{equation}
given by the pull-back of forms in each component. Now we take $U$ and $V$ open sets of $M_1$ and we tensor with the fixed vector space $H^{n-p}(M_2, \SJ_2)$ the associated Mayer-Vietoris sequence to obtain
\begin{eqnarray*}
\cdots &\to & H^p(U \cup V, \SJ_1) \otimes H^{n-p}(M_2, \SJ_2) \\
& \to & (H^p(U , \SJ_1) \otimes H^{n-p}(M_2, \SJ_2)) \oplus (H^p(V, \SJ_1) \otimes H^{n-p}(M_2, \SJ_2)) \\
& \to & H^p(U \cap V, \SJ_1) \otimes H^{n-p}(M_2, \SJ_2) \to \cdots.
\end{eqnarray*}
Summing up for $p=0, \ldots, n$, we obtain the exact sequence
\begin{eqnarray*}
\cdots &\to & \bigoplus_{p=0}^n H^p(U \cup V, \SJ_1) \otimes H^{n-p}(M_2, \SJ_2) \\
& \to & \bigoplus_{p=0}^n (H^p(U , \SJ_1) \otimes H^{n-p}(M_2, \SJ_2)) \oplus (H^p(V, \SJ_1) \otimes H^{n-p}(M_2, \SJ_2)) \\
& \to & \bigoplus_{p=0}^n H^p(U \cap V, \SJ_1) \otimes H^{n-p}(M_2, \SJ_2) \to \cdots
\end{eqnarray*}
to which we apply the morphisms \eqref{eq:Ku} to obtain the following commutative exact diagram
$$ \tiny{
\begin{array}{ccc}
\vdots & & \vdots \\
\downarrow & & \downarrow \\
\bigoplus_{p=0}^n H^p(U \cup V, \SJ_1) \otimes H^{n-p}(M_2, \SJ_2) & \stackrel{\Psi}{\to} & H^n((U\cup V)\times M_2, \SJ_{12})  \\
\downarrow & & \downarrow \\
\bigoplus_{p=0}^n (H^p(U , \SJ_1) \otimes H^{n-p}(M_2, \SJ_2)) \oplus (H^p(V, \SJ_1) \otimes H^{n-p}(M_2, \SJ_2)) & \stackrel{\Psi}{\to} & H^n(U\times M_2, \SJ_{12})\oplus H^n(V\times M_2, \SJ_{12}) \\
\downarrow & & \downarrow \\
\bigoplus_{p=0}^n H^p(U \cap V, \SJ_1) \otimes H^{n-p}(M_2, \SJ_2) & \stackrel{\Psi}{\to} & H^n((U\cap V)\times M_2, \SJ_{12})  \\
\downarrow & & \downarrow \\
\bigoplus_{p=0}^{n+1} H^p(U \cup V, \SJ_1) \otimes H^{n+1-p}(M_2, \SJ_2) & \stackrel{\Psi}{\to} & H^{n+1}((U\cup V)\times M_2, \SJ_{12})  \\
\downarrow & & \downarrow \\
\vdots & & \vdots \\
\end{array}}
$$
The commutativity is obvious except for the block
$$ \small{
\begin{array}{ccc}
\bigoplus_{p=0}^n H^p(U \cap V, \SJ_1) \otimes H^{n-p}(M_2, \SJ_2) & \stackrel{\Psi}{\to} & H^n((U\cap V)\times M_2, \SJ_{12})  \\
\downarrow \delta & & \downarrow \delta \\
\bigoplus_{p=0}^{n+1} H^p(U \cup V, \SJ_1) \otimes H^{n+1-p}(M_2, \SJ_2) & \stackrel{\Psi}{\to} & H^{n+1}((U\cup V)\times M_2, \SJ_{12})
\end{array}}
$$
that we check as follows. Let $\alpha \otimes \beta \in H^p(U \cap V, \SJ_1) \otimes H^{n-p}(M_2, \SJ_2)$. Denote the projections to the factors of the cartesian product $M_1\times M_2$ as $\pi_1$ and $\pi_2$. Therefore
$$ \Psi \circ \delta (\alpha \otimes \beta)= \pi_1^*(\delta \alpha) \wedge \pi_2^*(\beta)$$
and
$$ \delta \circ \Psi (\alpha \otimes \beta)= \delta (\pi_1^*(\alpha) \wedge \pi_2^*(\beta)).$$
We have to check that they are equal and a simple computation shows that it is true, just recalling that the form $\beta$ is closed.\footnote{For this we can follow exactly the same argument of \cite{BT82} (page 50) which consists in picking a partition of unity subordinated to the open sets $U$ and $V$.}

We now conclude the proof of Künneth formula in the same way Bott and Tu do it for the De Rham cohomology case (page 50 of \cite{BT82}).

Observe that because of lemma \ref{prop:coba}, K\"{u}nneth formula holds for $U$ and $V$ from the previous argument, K\"{u}nneth formula also holds for $U\cap V$, thus because of the Five Lemma it also holds for $U\cup V$. The K\"{u}nneth formula now follows by induction on the cardinality of a good cover (or a finite covering in the compact case). This concludes the proof of the theorem.
\end{proof}

\section{Applications I: The case of regular fibrations }
In this section we apply Mayer-Vietoris and K\"{u}nneth formula to compute Geometric Quantization of regular fibrations.

As an application of the previous formalism, we now give a simple proof of the quantization of a real polarization given by a regular fibration. This is exactly the case studied by \'{S}niaticky in \cite{sniatpaper}. Let us point out that a different proof was obtained by Hamilton in \cite{mhthesis} where a C\^{e}ch approach was used to deal with the general toric case (thus real polarizations given by integrable systems also admitting elliptic singularities).\footnote{A different proof of this result seen from the Poisson perspective  is a joint work of the first author of this paper with Mark Hamilton \cite{hamiltonmiranda2}.}

In order to do this, we will apply a K\"{u}nneth argument that will allow to prove the result by recursion on the following Lemma which addresses  the $2$-dimensional case.

\begin{lemma} \label{lem:tira}
Fix the domain $W= (-\epsilon, \epsilon) \times S^1$  endowed with a symplectic structure of integer class $\omega$ and consider as real polarization  the vertical foliation by circles. Then,

\begin{enumerate}
 \item In case there are no Bohr-Sommerfeld leaves then the Geometric Quantization is zero.
 \item If there is one Bohr-Sommerfeld leaf, then the Geometric Quantization is given by $H^0(W, \SJ)=0$ and $H^1(W, \SJ)= \C$.

     \end{enumerate}
\end{lemma}
\begin{proof}
Take coordinates $(x, \theta)\in (-\epsilon, \epsilon) \times S^1$. Recall that $\omega= f(x,\theta)dx \wedge d\theta$, for some $f(x,\theta) >0$. Fix $S^1=\R/\Z$. The form $\lambda= (\int_{0}^{x} f(t,\theta)dt) d\theta$ is a primitive for $\omega$ in the open set $(-\epsilon, \epsilon) \times (0,1)$. So $\lambda=h(x,\theta) d\theta$, where $\frac{dh}{dx}>0$. Let us take a system of coordinates over $U$ in which we can trivialize the prequantization bundle in the horizontal directions starting with a parallel section along the vertical axis. In this bundle trivialization, the connection $1$-form for the prequantizable bundle is expressed as,
$$ \nabla = d -2\pi i h(x,\theta)d\theta. $$

Let us denote by $s_0$ this trivializing section. Any parallel section $s:W \to L\simeq W \times \C$ can be expressed in terms of this trivializing section by means of the following formula,
\begin{equation}
s(x,\theta)= f(x,1/2) e^{\int_{1/2}^{\theta} 2\pi ih(x,s)ds} s_0. \label{eq:flat}
\end{equation}

From now on,  and for the sake of simplicity, we will identify the section $s_U(x_0,\theta_0)$ in a given neighbourhood with a function $f(x_0,\theta_0)$ via this trivializing section.

Now we take a covering of the domain $W$ by the pair of open sets
$$ U = (-\epsilon, \epsilon) \times (-0.1, 0.6). $$
$$ V = (-\epsilon, \epsilon) \times (0.4, 1.1). $$
The intersection is the domain
$$ U \cap V = (-\epsilon, \epsilon) \times ((-0.1, 0.1) \cup (0.4, 0.6))= W_1 \cup W_2. $$
Lemma \ref{prop:coba} entails that $H^1(D, \SJ)=0$ and $H^0(D, \SJ)=C^{\infty}((-\epsilon, \epsilon), \C)$ for any of the domains $D\in\{U, V, W_1, W_2\}$ because all of them are cotangent balls.

\begin{figure}[ht]
\includegraphics[scale=0.3]{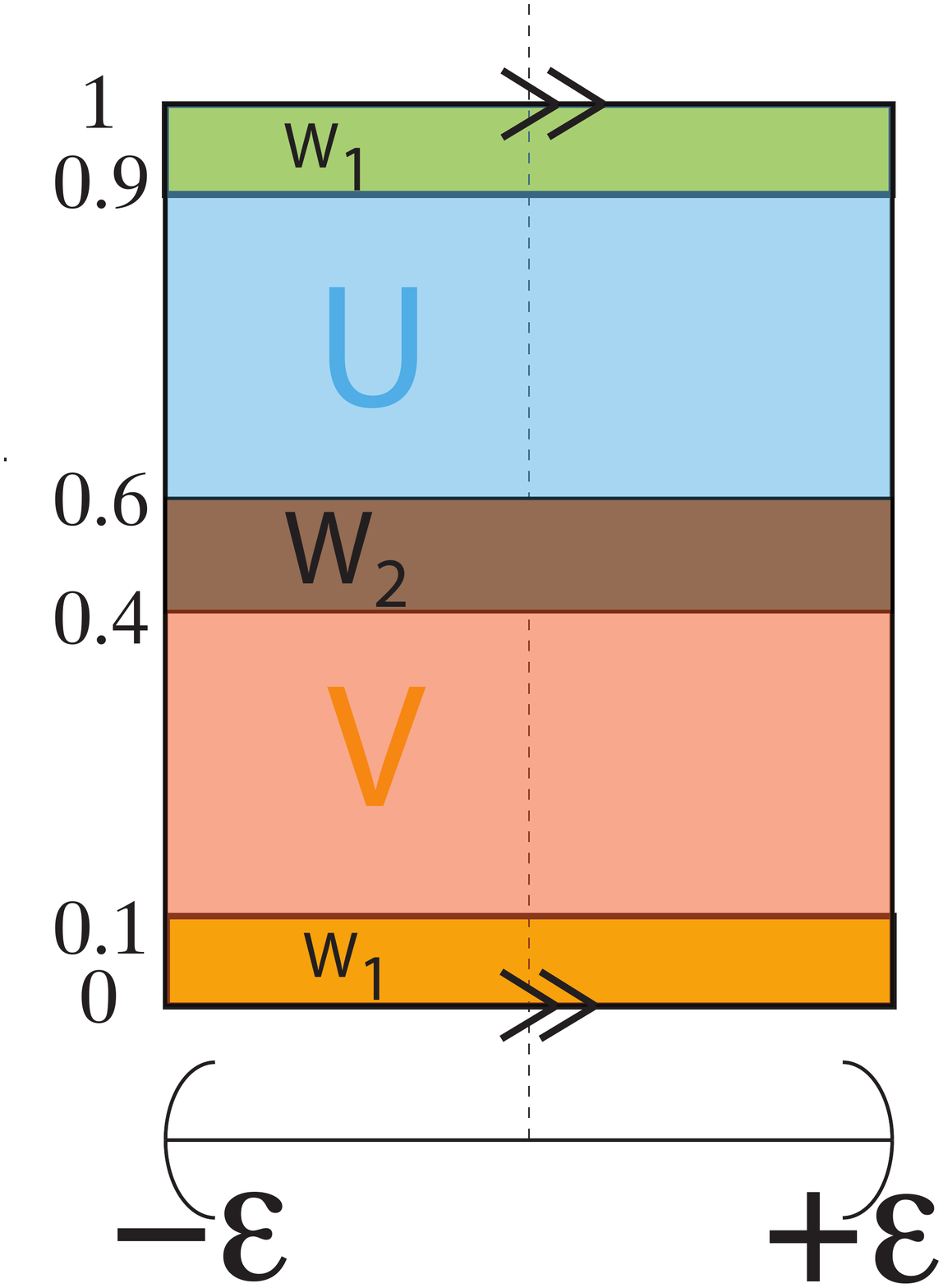}
\caption{Picture of the Mayer-Vietoris covering.}\label{fig:uv}
\end{figure}

We apply the Mayer-Vietoris sequence \eqref{eq:MV} to the pair $U,V$ to obtain
$$
%\SMALL{
\xymatrix{
 0 \ar[r] & C^{\infty}((-\epsilon, \epsilon), \C)\oplus C^{\infty}((-\epsilon, \epsilon), \C) \ar[r]^-{r_1-r_2} & C^{\infty}((-\epsilon, \epsilon), \C) \oplus C^{\infty}((-\epsilon, \epsilon), \C) \ar[dll] \\
H^1(M, \SJ) \ar[r] &  0 \ar[r] & 0
}
%}
$$

Observe that $H^0(W, \SJ)=0$ because there are no global flat sections since the space of closed non Bohr-Sommerfeld orbits is dense and over them any global section has to vanish.

 Let us now compute $H^1(W, \SJ)$;  The exactness of the sequence yields,
$$H^1(W, \SJ)= \frac{(C^{\infty}((-\epsilon, \epsilon), \C) \oplus C^{\infty}((-\epsilon, \epsilon), \C))} {  Im (r_1-r_2)} $$

So we just need to examine the image of $r_1-r_2$. In order to do that,
take sections $s_U \in H^0(U, \SJ)$, $s_V \in H^0(V, \SJ)$, $s_{W_1} \in H^0(W_1, \SJ)$ and
$s_{W_2} \in H^0(W_2, \SJ)$.

 From formula \eqref{eq:flat} it is easily seen that the restrictions of the sections $s_U(x, 0.5)$, $s_V(x, 0.5)$, $s_{W_2}(x, 0.5)$ and $s_{W_1}(x,0)$ completely determine the sections. Also observe that we are identifying the sections with functions with the help of a local trivializing section $s_0$.
  Thus the space of global flat sections of these domains is $C^{\infty}((-\epsilon, \epsilon), \C)$. Now we want to determine the image of the morphism $r_1-r_2$. Choose an element $(s_{W_1}, s_{W_2})$  of the space $H^0(W_1, \SJ) \oplus H^0(W_2, \SJ)$. We look for sections $s_U$ and $s_V$ such that the image under $r_1-r_2$ are the given pair. For this it is necessary to impose,
\begin{equation}
s_{W_1}(x, 0.5)= s_U(x, 0.5) -s_V(x, 0.5). \label{eq:im1}
\end{equation}
Using the equation \eqref{eq:flat} we obtain the following condition
\begin{equation}
s_{W_2}(x, 0)= s_U(x, 0.5)e^{\int_{0.5}^{0} 2\pi ih(x,s)ds}  -s_V(x, 0.5)e^{\int_{0.5}^{1} 2\pi ih(x,s)ds}. \label{eq:im2}
\end{equation}
The equations \eqref{eq:im1} and \eqref{eq:im2} provide a family of systems of two equations and two variables depending on the parameter $x\in (-\epsilon, \epsilon)$. In order for them to have solution for every value of the pair $(s_{W_1}, s_{W_2})$, the following conditions need to hold,
$$ \left| \begin{array}{cc}
1 & -1 \\
e^{\int_{0.5}^{0} 2\pi ih(x,s)ds} & -e^{\int_{0.5}^{1} 2\pi ih(x,s)ds}
\end{array} \right| \neq 0.
$$
Observe that this is not satisfied exactly for the values of $x$ lying in a  Bohr-Sommerfeld orbit.

 Summing up, this implies,

 \begin{enumerate}

 \item In the case there are no Bohr-Sommerfeld leaves the mapping $r_1-r_2$ is surjective and therefore $H^1(W, \SJ)=0$.

 \item In the case there is one Bohr-Sommerfeld leaf, the sections $s_{W_1}$ and $s_{W_2}$ are related by one equation; the image of $r_1-r_2$ has codimension $1$ and $H^1(W, \SJ)=\mathbb C$.

     \end{enumerate}
   and this finishes the proof of the Lemma.
\end{proof}

Now, the classical \'{S}niaticky theorem becomes a simple consequence of this Lemma and K\"{u}nneth formula. Let $\SF$ be a Lagrangian foliation on a manifold $(M, \omega)$. We say that a torus $\mathbb T^n$, that is a leaf of the foliation, admits a cotangent neighbourhood if there is a neighborhood $U$ of $\mathbb T^n$ and a symplectic chart $\phi: U \to V \subset \mathbb T^n\times \R^n$ such that
\begin{itemize}
\item The Lagrangian foliation becomes the  foliation given by leaves $\mathbb T^n \times \{ p \}$.
\item There are $(x_1, \ldots, x_n)$ global affine coordinates in $\mathbb T^2$ and $(y_1, \ldots, y_n)$ the standard coordinates of $\R^n$, such that the symplectic form $\omega_{\phi}= \phi^*( \omega)$ can be written as
$$ \omega_{\phi} (x_1, \ldots, x_n, y_1, \ldots, y_n) = \Sigma_{i=1}^{n}dx_i \wedge dy_i. $$
\end{itemize}
As an example, any compact fiber of an integrable systems admits a cotangent neighbourhood (the fiber is a Liouville torus and the foliation is semilocally a fibration by tori \cite{arnold, duistermaat}).
We can now prove the following theorem which is a reformulation of  \'{S}niaticky's theorem.

\begin{theorem}[\'{S}niaticky]
Let $(M^{2n}, \omega)$ be a symplectic manifold, $N$ a smooth manifold and $\pi: M \to N$ a fibration by tori. Assume that all the fibers admit cotangent neighborhoods. Moreover assume that the number of Bohr-Sommerfeld fibers is finite and equal to $k$. Then,
\begin{enumerate}
\item $H^i(M, \SJ)= 0$, for all $i \neq n$.
\item $H^n(M, \SJ)= \C^k$.
\item $\SQ (M, \SP)= \C^k$.
\end{enumerate}
\end{theorem}
\begin{proof}
Take a locally finite covering $\{ V_i \}_{i\in \N}$ of $N$ in such a way that the image by $\pi$ of each Bohr-Sommerfeld leaf is covered by a single open set. Denote by $U_i= \pi^{-1} (V_i)$ that is a locally finite covering of $M$. Without any loss of generality we may assume that the first $k$ indices correspond to the open sets containing the Bohr-Sommerfeld orbits.

We claim that $\SQ (U_j, \SP)=0$ for any $j>k$ and $\SQ(U_i, \SP)= \C$ for $i\leq k$. This is because the
neighborhood $(U_i, \SJ_i)$ trivializes, using the cotangent neighborhood trivialization, as:
$$ (U_i, .\SJ_i) = \prod_{l=1}^n ((-\epsilon, \epsilon) \times S^1, \tilde{\SJ}), $$
where the sheaf $\tilde{\SJ}$ is the one associated to the Lagrangian foliation defined in the statement of Lemma \ref{lem:tira}. We are under the hypotheses of K\"unneth Theorem (Theorem \ref{thm:Kun}) and we can apply induction and K\"{u}nneth formula  to prove that the Quantization of the open sets $U_i$ is the one stated above.

Last but not least, we need to perform a Mayer-Vietoris argument to glue these neighbourhoods.
To that end, define $M_j= \bigcup_{i=1}^j U_i$; the Mayer-Vietoris sequence yields,
\begin{enumerate}
\item $H^i(M_k, \SJ)= 0$, for all $i \neq n$.
\item $H^n(M_k, \SJ)= \C^k$.
\item $\SQ (M_k, \SP)= \C^k$.
\end{enumerate}
Moreover,  the Mayer-Vietoris sequence also yields  $\SQ(M_j, \SP)= \SQ(M_l, \SP)$, for any $j,l \geq k$. This concludes the proof.
\end{proof}

\section{Applications II: Irrational foliation of the $2$-torus.}\label{irrational}

In this section we apply Mayer-Vietoris and K\"{u}nneth formula to compute Geometric Quantization of the irrational flow on the torus.

Let $\mathbb T^2= \R^2/\Z^2$ be the $2$-torus with coordinates $(x,\theta)$. Define the vector field $X_{\eta}= \eta \frac{\partial}{\partial x}+ \frac{\partial}{\partial \theta}$, with $\eta\in \R$. This vector field descends to a vector field in the quotient torus that we still denote $X_{\eta}$. Denote by $\SP_{\eta}$ the associated foliation in $\mathbb T^2$. If $\eta$ is an irrational number, the foliation $\eta$ is named the irrational foliation of the torus with slope $\eta$. It is well-known \cite{denjoy,kneser} that any foliation of the torus without periodic orbits is topologically conjugate to $\SP_{\eta}$ for some irrational $\eta$. So the next result computes the Geometric Quantization of the $2$-torus polarized by any regular foliation without periodic orbits up to topological equivalence. The metalinear bundle $N$ is trivial in this case.
\begin{theorem}
Let $(\mathbb T^2, \omega)$ be the $2$-torus with the standard symplectic structure $\omega=p dx\wedge d\theta$ of area $p\in \N$ and let $\SP_{\eta}$ the irrational foliation of slope $\eta$ in this manifold. Then
\begin{enumerate}
\item The Geometric Quantization space is infinite dimensional.
\item The foliated cohomology space $\SQ(\mathbb T^2, \SJ)$ is infinite dimensional if the irrationality measure of $\eta$ is infinite (i.e. it is a Liouville number). If the irrationality measure is finite then $\SQ(\mathbb T^2, \SJ)= \C \bigoplus \C$.
\end{enumerate}
\end{theorem}

\begin{remark}
It is interesting to point out here as it was done by Heitsch in \cite{heitsch} in the case $\eta$ is not a Liouville number, there is a nice interpretation of the elements of $H^1(M,\SJ)$ in the limit case of foliated cohomology as the associated infinitesimal deformations of a differentiable family of foliations.
\end{remark}
\begin{proof}
Take coordinates $(x, \theta)\in S^1 \times S^1$. Recall that $\omega= kdx \wedge d\theta$, for some $k>0$. Let $\lambda_{\eta}=  d\theta-  \frac{1}{\eta} dx$ be the $1$-form defining the Lagrangian foliation.

%By Poincar\'e lemma, there exists a primitive of the form $\alpha = h(x,\theta) \lambda_{\eta}$ for the symplectic form $\omega$ in the open set $W= (0,1) \times (0,1)$. Let us take a system of coordinates over $W$ in which we can trivialize the prequantization bundle in the horizontal directions starting with a parallel section along the vertical axis.
%The connection $1$-form for the prequantizable bundle
%, induced by parallel transport along the Lagrangian foliation, is given by
%$$ \nabla = d -2\pi i \alpha. $$

%and having chosen a trivializing section $s_0$, any parallel section $s:W \to L\simeq W \times \C$ is of the form
%\begin{equation}
%s(x_0+ \eta \epsilon, 1/2+ \epsilon) = f(x_0,1/2) e^{\int_{\SL_0} 2\pi i\alpha} s_0, \label{eq:flat2}
%\end{equation}
%for $\SL_0$ the leaf connecting the points $(x_0, 1/2)$ and $(x_0+\eta\epsilon, 1/2 +\epsilon)$ and contained inside this neighbourhood.
%For the sake of simplicity, we are identifying the sections with functions in this neighbourhood (since formula \ref{eq:flat2} allows us to do so in this neighbourhood).

In the case of  foliated cohomology $\alpha=0$ and therefore the parallel transport equation implies that
$$s(x, 1/2)=s(x+ k/\eta, 1/2), $$
for all $k \in \mathbb Z$. Therefore the section is constant in a dense set and therefore constant because of the irrationality of $\eta$. Thus $H^0(\mathbb T^2, \SJ)= \C$. For the Geometric Quantization case, assume that there is a section $\sigma \in H^0(\mathbb T^2, \SJ)$. Therefore, since the prequantizable bundle is topologically non-trivial we have that there is a point $p\in T^2$ such that $\sigma(p)=0$. The parallel transport along the leaf of the foliation containing $p$ allows to conclude that the section vanishes along the leaf. Since, the leaf is dense we obtain that $\sigma=0$. So, we have that $H^0(\mathbb T^2, \SJ)= 0$ as claimed.

As in previous Sections, we cover $\mathbb T^2$ by two open sets $U= S^1 \times (-0.1, 0.6)$ and $V= S^1 \times (0.4, 1.1)$. The intersections are given by
\begin{eqnarray*}
W_1 & = & S^1 \times (-0.1, 0.1), \\
W_2 & =  & S^1 \times (0.4, 0.6).
\end{eqnarray*}

To compute the Geometric Quantization of $(\mathbb T^2, \SP_{\eta})$ we use the Mayer-Vietoris sequence applied to $U$ and $V$.  As we have already proved in Section \ref{sec:functorial}, $H^0(S^1 \times (-\delta, \delta), \SJ)= C^{\infty}(S^1, \C)$ and $H^1(S^1 \times (-\delta, \delta), \SJ)=0$. Thus the sequence becomes,
$$
\xymatrix{
0 \ar[r] & C^{\infty}(S^1, \C)\oplus C^{\infty}(S^1, \C) \ar[r]^-{r_1-r_2} & C^{\infty}(S^1, \C) \oplus C^{\infty}(S^1, \C) \ar[dll] \\
H^1(\mathbb T^2, \SJ) \ar[r] & 0 \ar[r] & 0
}
$$

Now we have sections $s_U(x, \theta)$, $s_V(x, \theta)$, $s_1(x,\theta)$, $s_2(x,\theta)$ over $U$, $V$, $W_1$ and $W_2$ respectively. By the parallel transport equation these sections are completely determined by the restrictions $s_U(x, 0.5)$, $s_V(x, 0.5)$, $s_1(x,0)$, $s_2(x,0.5)$. We want to solve the equations,

\begin{equation}
s_{2}(x, 0.5)= s_U(x, 0.5) -s_V(x, 0.5). \label{eq:im1}
\end{equation}
Using  \eqref{eq:flat}  the following condition is obtained,
\begin{equation}
s_{1}(x, 0)= s_U(x,0)  -s_V(x,0). \label{eq:im2}
\end{equation}

Let us start by solving the foliated cohomology case. In that case the sections are functions and the parallel transport implies just that the functions are constant along each leaf.  The equation (\ref{eq:im2}) reads
\begin{equation}
s_1(x, 0)= s_U(x+\frac{\eta}{2}, 0.5) -s_V(x-\frac{\eta}{2}, 0.5). \label{eq:im3}
\end{equation}
We denote $w(x)= s_{W_1}(x,0)$, $\hat{w}(x)=s_{W_2}(x, 0.5)$, $u(x)=s_U(x,0.5)$ and $v(x)=s_V(x, 0.5)$. Since they are smooth functions over the circle, they admit a Fourier series representing them. The coefficients associated to the previous $4$ functions are ${w_k}$, ${\hat{w}_k}$, ${u_k}$ and ${v_k}$ respectively. By Theorem \ref{thm:Graf} the decay of the Fourier coefficients of $u_k$ and $v_k$ satisfies for any fixed positive integer $q>0$ the following relations,
\begin{eqnarray}
\lim_{k \to \pm \infty} \frac{|u_k|}{1+ |k|^q} & = & 0 \label{eq:decayu}, \\
\lim_{k \to \pm \infty} \frac{|v_k|}{1+ |k|^q} & = & 0 \label{eq:decayv}.
\end{eqnarray}
The equations (\ref{eq:im1}) and (\ref{eq:im3}) imply the following set of equations in the Fourier coefficients:
\begin{equation*}
\left( \begin{array}{cc} \hat{w}_k \\ w_k \end{array} \right) = \left( \begin{array}{cc} 1 & -1 \\ e^{\pi i k \eta} & -e^{-\pi i k \eta} \end{array} \right) \cdot \left( \begin{array}{cc}u_k \\ v_k \end{array} \right), \hspace{0.3cm} k\in \Z.
\end{equation*}
Since the number $\eta$ is irrational the equation has a unique solution expressed as
\begin{equation}
\left( \begin{array}{cc}u_k \\ v_k \end{array} \right) = \left( \begin{array}{cc} -e^{-\pi i k \eta}  & 1 \\ -e^{\pi i k \eta} & 1 \end{array} \right) \frac{1}{e^{\pi i k \eta} - e^{-\pi i k \eta}} \left( \begin{array}{cc} \hat{w}_k \\ w_k \end{array} \right), \hspace{0.3cm} k\in \Z- \{ 0\}. \label{eq:solution}
\end{equation}
In particular, when $\hat{w}_0\neq w_0$ there is no solution.

Thus, in order to have a solution from now on we will assume  $\hat{w}_0= w_0$.

If $\nu \geq 2$, the irrationality measure of $\eta$, is finite (i. e. it is not a Liouville number),  we obtain
\begin{equation}
| \eta - \frac{p}{k}| \geq \frac{1}{k^{\nu}}, \label{eq:measure}
\end{equation}
for any pair $(p,k)$ with $k$ large enough. From there we obtain
$$ |k\eta -p| \geq \frac{1}{k^{\nu +1}}. $$
Therefore, we have
\begin{equation}
\left| \frac{1}{e^{\pi i k \eta} - e^{-\pi i k \eta}} \right| = \frac{1}{|e^{2\pi  i k \eta}-1|} \leq \frac{1}{|k\eta- p|} \leq k^{\nu+1}, \label{eq:master_in}
\end{equation}
the first inequality comes from the inequality
$$ |e^{2\pi i t}- e^{2\pi i s}| \leq | t-s |.$$
We easily obtain from the equations (\ref{eq:solution}), (\ref{eq:decayu}), (\ref{eq:decayv}) and the inequality (\ref{eq:master_in}) the following limits, for any $q\geq0$:
\begin{eqnarray*}
\lim_{k \to \pm \infty} \frac{|\hat{w}_k|}{1+ |k|^q} & = & 0,  \\
\lim_{k \to \pm \infty} \frac{|w_k|}{1+ |k|^q} & = & 0.
\end{eqnarray*}
Now, we have bounded the decay of the coefficients of the solutions $u(x)$ and $v(x)$. We are under the hypotheses of the Theorem \ref{thm:Graf} to conclude that $u$ and $v$ are smooth functions on the circle. Therefore  going back to the long exact sequence, the map $r_1-r_2$ is surjective when $w_0=\hat{w}_0$ yielding  $H^1(\mathbb T^2, \SJ)=\C$.

Now when of $\eta$ is a Liouville number we may assume that, there exists a sequence of pairs of positive integers $\{ p_s, k_s \}_{s\in \Z^*}$ such that
$$ | \eta - \frac{p_s}{k_s} | < \frac{1}{k^s}, $$
satisfying that $k_s$ is strictly increasing. From this equation we obtain arguing as in the previous case that,
\begin{equation}
\frac{1}{|e^{2\pi i k\eta}-1|} > \frac{k_s^{s-1}}{2\pi}. \label{eq:bound}
\end{equation}
Define the constant sequence $a(0)_k= 1$. Now, we define the functions $\hat{w}=0$ and $w$ with Fourier coefficients
$$
w_l= \left\{ \begin{array}{ll} \frac{a(0)_k}{k_s^{s-1}}, & {\rm if} \, \exists s \in \Z^+, {\rm \,such \,that} \,k_s=l, \\ 0, & {\rm otherwise. }\end{array} \right.
$$
 Notice that because of Theorem  \ref{thm:Graf}, the function $w$ is smooth. However, by combining the equation  (\ref{eq:solution}) and the inequality (\ref{eq:bound}) we obtain
$$|v_{k_s}| \geq \frac{1}{2\pi}.$$
Therefore, the function $v$, solving equations (\ref{eq:im1}) and (\ref{eq:im3}) for the selected input data $\hat{w}$ and $w$, is not smooth according to Theorem \ref{thm:Graf}  since its Fourier coefficients do not converge to zero for $k$ large.

We have shown that the map $r_1-r_2$ is not surjective. Therefore the cokernel of the map is not zero. To show that the cokernel is infinite dimensional we define the family of sequences $a(r)$ defined as follows:
$$
a(r)_k = \left\{ \begin{array}{ll} 1, & k\neq r, \\ 2, & k=r \end{array} \right.
$$
Then, we choose input data $\hat{w}=0$ and $w$  a smooth function with the following  prescribed Fourier coefficients,
$$
w_l= \left\{ \begin{array}{ll} \frac{a(r)_k}{k_s^{s-1}}, & if \, \exists s \in \Z^+, \,such \,that \,k_s=l, \\ 0, & otherwise. \end{array} \right.
$$
The linear expand of the set of choices determines an infinite dimensional subspace of functions and so it shows that the cokernel of $r_1 -r_2$ is infinite dimensional. This implies that the dimension of $H^1(\mathbb T^2, \SJ)$ is infinite.

In the case of Geometric Quantization, we assume that $\omega=p dx\wedge d\theta$, for a fixed positive integer $p$. We fix two different trivialization sections $\sigma_U$ and $\sigma_V$ over the regions $U$ and $V$. Without loss of generality we assume that the circle $S^1 \times \{ 0 \}= S_0$ is a Bohr-Sommerfeld submanifold (for the horizontal polarization) with fixed global section over it. They are defined by the condition that they coincide over $S_0$ with the Bohr-Sommerfeld section and that they are defined by parallel vertical transport (upwards for $\sigma_U$ and downwards for $\sigma_V$). They satisfy the gluing equation
\begin{equation}
\sigma_U(x, 0.5)= \sigma_V(x, 0.5) e^{2\pi i p x}. \label{eq:gluing}
\end{equation}
With respect to any of these two trivializing sections the connection of the prequantizable bundle is written as
\begin{equation}
\nabla = d +2\pi i p \theta dx. \label{eq:connec}
\end{equation}
Use the following notation
\begin{enumerate}
\item $s_U(x, \theta)=f_U(x, \theta) \cdot \sigma_U(x, \theta)$,
\item $s_V(x, \theta)=f_V(x,\theta) \cdot \sigma_U(x, \theta)$,
\item $s_2(x, \theta)=f_2(x, \theta) \cdot \sigma_U(x, \theta)$,
\item $s_1(x, \theta)=f_1(x,\theta) \cdot \sigma_U(x, \theta)$.
\end{enumerate}
Dividing by $\sigma_U(x, 0.5)$ on both sides of equation (\ref{eq:im1}) we obtain
$$
f_{2}(x, 0.5)= f_U(x, 0.5) -f_V(x, 0.5).
$$
Now, by using the connection formula (\ref{eq:connec}) over $U$ we obtain, by parallel transport along the leaves, the following formula
\begin{eqnarray*}
f_U(x+ \eta \theta, \theta) &= & f_U(x,0) \cdot e^{2\pi i p \theta/\eta}.
%f_1(x+ \eta \theta, \theta) &= & f_1(x,0) \cdot e^{2\pi i p \theta/\eta}.
\end{eqnarray*}
Thus, we obtain
\begin{eqnarray}
f_U(x,0) &= & f_U(x+ 0.5\eta, 0.5)\cdot e^{-\pi i p \theta/\eta}. \label{eq:changeU}
%f_1(x,0)  &= & f_1(x+ \eta 0.5, 0.5)\cdot e^{-\pi i p \theta/\eta}
\end{eqnarray}
Observe that
$$
f_V(x,\theta)= \frac{s_V(x,\theta)}{\sigma_U(x,\theta)} = \frac{s_V(x,\theta)}{\sigma_V(x,\theta)}\frac{\sigma_V(x,\theta)}{\sigma_U(x,\theta)}= \tilde{f}_V(x, \theta) e^{-2\pi i p x},
$$
where $\tilde{f}_V(x, \theta)= \frac{s_V(x,\theta)}{\sigma_V(x,\theta)}$. It satisfies the equation, by a computation analogous to the one providing (\ref{eq:changeU}),
\begin{eqnarray}
\tilde{f}_V(x,0) &= & \tilde{f}_V(x- 0.5\eta, 0.5)\cdot e^{\pi i p /\eta}. \label{eq:changeV}
%f_1(x,0)  &= & f_1(x+ \eta 0.5, 0.5)\cdot e^{-\pi i p \theta/\eta}
\end{eqnarray}

Therefore, after quotienting by $\sigma_U$,  equation (\ref{eq:im2})  becomes,
\begin{equation*}
f_1(x,0)= f_U(x,0) - f_V(x,0),
\end{equation*}
that is, by substituting in the previous equations
\begin{equation}
f_1(x,0)= f_U(x+ 0.5\eta, 0.5)\cdot e^{-\pi i p/\eta} - f_V(x- 0.5\eta, 0.5)\cdot e^{\pi i p /\eta}\cdot e^{2\pi i p x},
\label{eq:im4}
\end{equation}

We have smooth functions:
\begin{eqnarray*}
 \hat{w} & = & f_2(x,0.5), \\
w & = & f_1(x,0), \\
u & = & f_U(x, 0.5), \\
v & = & f_V(x, 0.5).
\end{eqnarray*}
They completely recover the initial four sections $s_U, s_V, s_1$ and $s_2$.
Expanding the Fourier coefficients of them and substituting them into the equations (\ref{eq:im1}) and (\ref{eq:im4}) we obtain the sequence of systems of equations
\begin{equation*}
\left\{ \begin{array}{ccc} \hat{w_k} & = & u_k -v_k \\ w_k & = & e^{-\pi i (k \eta +p/\eta)}u_k - e^{\pi i (k \eta +p/\eta)}v_{k-p} \end{array} \right. \, ,k\in \Z
\end{equation*}
Substituting the first equation in the second and simplifying we obtain
\begin{equation}
w_k - e^{-\pi i (k \eta +p/\eta)}\hat{w}_k = e^{-\pi i (k \eta +p/\eta)} v_k - e^{\pi i (k \eta +p/\eta)}v_{k-p}, k\in \Z.
\label{eq:compu}
\end{equation}
Let us compute a particular case. Assume that $\hat{w}=0$ and $w=1$. Therefore $w_k= \delta_{0k}$. This immediately implies that
$$
v_q= 0,$$
for $q \not \cong 0 \mod p$. For the multiples of $p$, we can choose any $v_0\in \C$. Substituting in (\ref{eq:compu}) for $k=0$, we obtain the value of $v_{-p}$. Using again equation (\ref{eq:compu}) for $k>0$, it is simple to check that
$$
|v_{k\cdot p}|=|v_0|,
$$
for $k>0$. Substituting for negative values we obtain
$$
|v_{k\cdot  p}|=|v_{-p}|,
$$
for $k< -1$. Therefore, at least one of the two limits of $v_k$ does not converge to zero. This implies by Theorem \ref{thm:Graf} that any of the solutions (one for each choice of $v_0$) of equation (\ref{eq:compu}) produces a function $v$ that is not smooth.

The previous choice of $w$ and $\hat{w}$ can be slightly perturbed to produce examples of pairs $(w,\hat{w})$ such that the associated solution $(u,v)$ is not a smooth pair. This shows that the dimension of the space $H^1(\mathbb{T}^2, \SJ)$ is infinite.
\end{proof}

\begin{remark}
Most certainly a similar argument yields the infinite dimensionality of the Geometric Quantization space for the case of a general symplectic form. It is just a matter of complicating the formulae, but the basic non-decaying behavior of the Fourier coefficients is probably kept.
\end{remark}

\section{Applications III: Geometric Quantization of general foliations over the $2$-torus} \label{sec:torus}

In this section we deal with the case of general real polarizations given by a non--singular flow on the torus. For this we first recover the topological classification and then apply the previous functorial properties to compute its Geometric Quantization. In the way, we also obtain some results for foliated cohomology associated to this foliation.
\subsection{The topological classification}
We recall here the topological Denjoy-Kneser \cite{denjoy,kneser,Re62} classification, up to topological equivalence, of line fields $\SF$ on the torus. Let us point out that the classifications that we use here is the $\mathcal{C}^{0}$ classification and not the $C^{1}$ classification. For instance the so-called Denjoy foliation (see for instance \cite{denjoyexample})\footnote{ The Denjoy example has been a key example in the theory of foliations. For instance, it was used by Paul Schweitzer to give a counterexample to the Seifert conjecture for closed orbits for $\mathcal{C}^1$ flows on $\mathbb S^3$ and indeed on any three dimensional manifold (see \cite{schweitzer}).} is not diffeomorphic to the irrational flow.  We follow the ideas and notation as presented in \cite{Re62}. A line field of the torus is completely determined by a map $L_{\SF}:\mathbb T^2 \to \RP^1$, i.e. we are trivializing the tangent bundle and taking its projectivization so each line becomes a point. If the line field is oriented then the maps lifts to $\hat{L}_{\SF}: \mathbb T^2 \to S^1$.

Recall that the standard obstruction theory establishes that the homotopy type of the line field, as a distribution, is determined by an obstruction class provided by
$$\lambda \in H^1(\mathbb T^2, \pi_1(\RP^1))=H^1(\mathbb T^2, \Z).$$
Fix a basis of $H_1(\mathbb T^2, \Z)= \Z^2$ provided by two elements represented as loops $\gamma_1$ and $\gamma_2$ that intersect transversely at the point $(0,0)\in \mathbb T^2 = \R^2 / \Z^2$. The maps $L_{\gamma_i}: S^1 \to S^1$, $i=1,2$, have  integer degrees $d_1$ and $d_2$ and $\lambda([\gamma_i])= d_i$. Therefore the homotopy classes of line fields are represented by a pair of integers that are the degrees of the restrictions of the map $L$ to a positive basis of $H_1(\mathbb T^2, \Z)$.  Moreover, it is obvious that the distribution is orientable if the degrees are even. Observe that the integers depend on the trivialization chosen on the torus, but for each fixed trivialization the classes run over all the possible pairs of integers.

The following lemma holds,

\begin{lemma}
A regular foliation $\SF$ over the torus $\mathbb T^2$ always admits a ``standard'' metaplectic correction.\end{lemma}
\begin{proof}
The real line bundles $\SF$ over $\mathbb T^2$ are completely classified by the first Stieffel-Withney class $w_1(\SF)\in H^1(\mathbb T^2, \Z_2)= (\Z_2)^2$. Therefore, there exist $4$ of them. We have the following exact sequence
$$ 0 \to \Z \stackrel{\cdot 2}{\to} \Z \to \Z_2 \to, 0 $$
that induces a long exact sequence
$$ \cdots H^1(\mathbb T^2, \Z) \to H^1(\mathbb T^2, \Z_2) \stackrel{\bf b}{\to} H^2(\mathbb T^2, \Z) \stackrel{\bf i}{\to} H^2(\mathbb T^2, \Z)\cdots, $$
where the map $\bf b$ is the Bockstein morphism that maps the first Stieffel-Whitney class of a real line bundle $w_1(\SF)$ to the first Chern class of its complexification $c_1(\SF \otimes_{\R} \C)$. The metaplectic correction gives a choice of square root of $\SF\otimes \C$. So, we just need to have $c_1(\SF \otimes_{\R} \C)\in H^2(\mathbb T^2, \Z)= \Z$ of even degree. But since ${\bf i}= 2\cdot id$ is injective, we obtain $c_1(\SF \otimes_{\R} \C)=0$. Therefore we have $\Lambda^1 (\SL \otimes_{\R} \C)= \C$ and the trivial bundle can be chosen as bundle of half-forms, i.e. the square root of the trivial line bundle is itself. Therefore the prequantizable bundle remains the same after the metaplectic correction and we call that choice the {\it standard} one.
\end{proof}

Now assume that there is a closed leaf of the line field $\SL$. It represents an element $[\SL]$ of $H_1(M, \Z)$. It is obvious that $\lambda([\SL])=0$. Any other closed leaf $\SL'$ represents the same homology class since it cannot intersect $\SL$ (since different orbits are disjoint) and by Poincar\'e-Bendixon Theorem cannot be null-homologous. Therefore, there are always homology classes which are not represented by closed leaves.

For any homology class $A\in H_1(M, \Z)$, we have the following result that is proved using  Sard's lemma.
\begin{lemma}
Fix $A\in H^2(M, \Z)$. There exists an embedded smooth loop $\gamma:S^1\to \mathbb T^2$ representing the class such that the number of tangencies of $\gamma$ and $\SF$ is finite.
\end{lemma}
Consequently, we can minimize that number obtaining the following,
\begin{definition}
A minimal contact curve $\gamma$ for a class $A\in H_1(M,\Z)$, not represented by closed curves, is a smooth curve that minimizes the number of tangencies with $\SF$.
\end{definition}
We obtain
\begin{corollary}
The leaf $\SL_0$ of $L$ at a tangency point $t_0$ of a curve of minimal contact does not cross the curve.
\end{corollary}
\begin{proof}
We can choose a small chart around $x_0$ such that $t_0=(0,0)\in \R^2$, the distribution $L$ is locally given by the equation \{ y=const \} and the curve $\gamma:(-\epsilon, \epsilon) \to \R^2$ is written as
$$ \gamma(t)= (t, f(t)),$$
with $f(0)=f'(0)=0$ and moreover we assume, by hypothesis, that $f(t)$ is increasing in a neighborhood of $t=0$. Then, it is simple to locally perturb (see figure \ref{fig:deform}) $f$ to a new $g$ such that $\hat{\gamma}(t)=(t,g(t))$ has no tangency points in the neighborhood of $(0,0)$. This implies that $\hat{\gamma}$ has less tangencies than $\gamma$ and therefore the initial $\gamma$ was not a minimal contact curve.
\end{proof}

\begin{figure}[ht]
\includegraphics[scale=0.3]{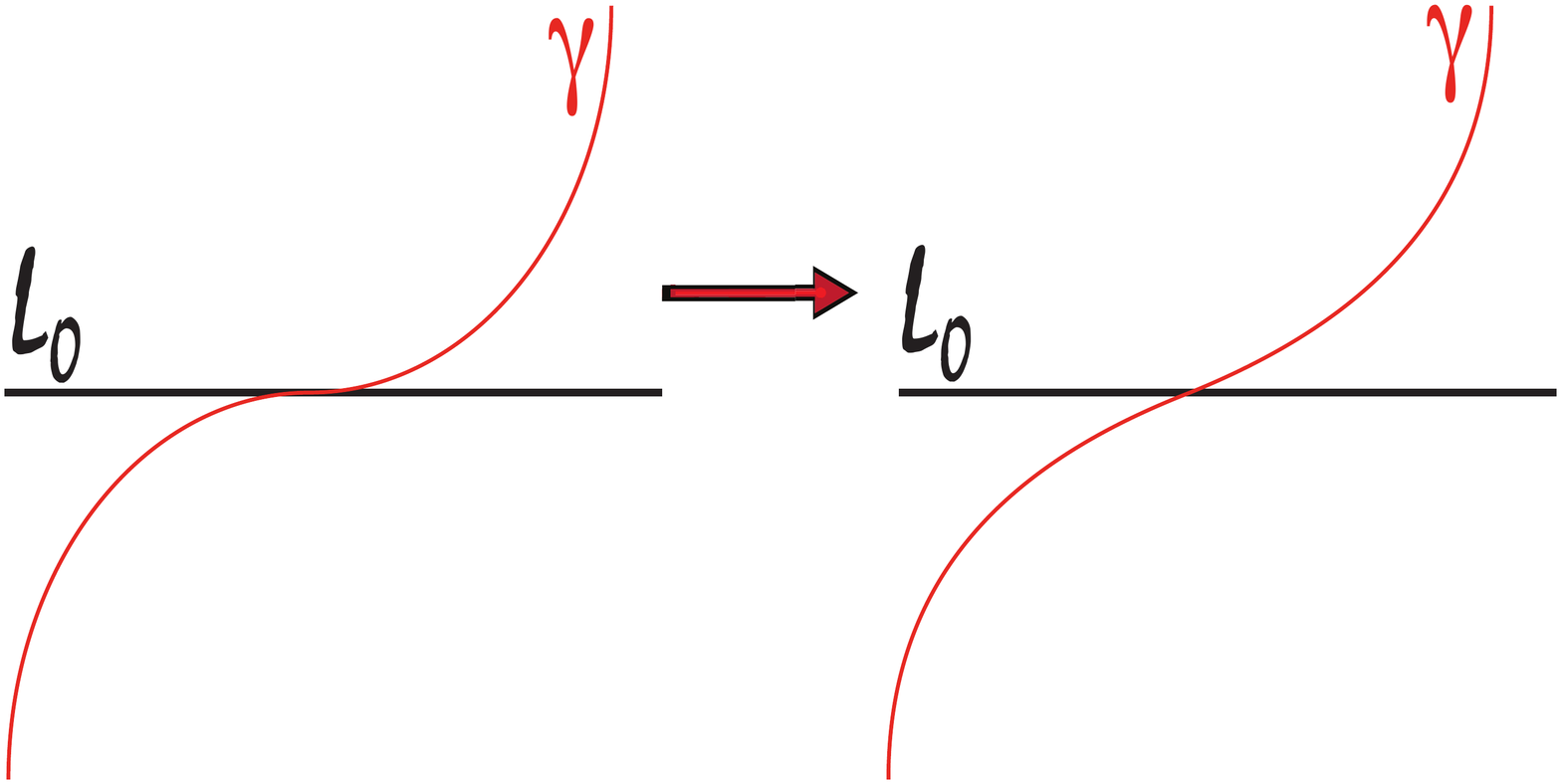}
\caption{Deformation of $\gamma$ such that $\hat{\gamma}(t)=(t,g(t))$ has no tangency points.}\label{fig:deform}
\end{figure}

We say that a tangency of a minimal contact curve $\gamma$ is positive if the map $L_{\gamma}:S^1 \to S^1$ is increasing at the tangency point, it is negative if it is decreasing. Assume that the homology class representing the closed leaves is the class $(0,1)\in \Z^2=H_1(\mathbb T^2, \Z)$. This can be always satisfied by composing the $\mathbb T^2 = \R^2/\Z^2$ torus by an element in $SL(2, \Z)$. The minimal contact curve can be assumed to lie in the class $(1,0)$. Now, any open leaf is diffeomorphic to the real line. It is known that the semiline (semiorbit) of the open leaf is asymptotically tangent to a closed orbit. Therefore, denoting by $\gamma:\R \to \mathbb T^2$ a parametrization of the line and $\bar{\gamma}=(\bar{\gamma}_1, \bar{\gamma}_2): \R \to \R^2$ the canonical lift to the universal cover, then the following limit exists
$$
\lim_{t \to \infty} \bar{\gamma}_2(t), $$
and it is either $+\infty$ or $-\infty$. We say that the semiorbit is positive if the limit is $+\infty$ and negative otherwise.

We obtain,
\begin{theorem}[Proposition 3 in \cite{Re62}] \label{thm:Rei}
Let $\Gamma$ of minimal contact have $\mu^+$ positive tangencies and $\mu^-$ negative tangencies. Then the leaves of the foliation can be described as follows:
\begin{enumerate}
\item There are $\mu^+$ regions bounded by a pair of (possibly not distinct) closed leaves, such that all the leaves interior to this region are open and have both semiorbits negative.
\item There are $\mu^-$ regions with both semiorbits positive.
\item There are some other regions (at most numerable in number) in which the two semiorbits of a given orbit have opposite signs.
\item If the complement of these regions is not the whole space, it is composed entirely of closed leaves.
\item If the complement is the whole space either all the leaves are closed, or all are dense.
\end{enumerate}
\end{theorem}

A foliation is called {\emph{generic}} if the linear monodromy of the closed leaves is non-degenerate (not the identity), i.e. the linearized Poincar\'e return map is of the form $p(t)= \lambda t$, with $\lambda>0$ and different from $1$. This, in particular, implies that the closed leaves are isolated and stable under $C^1$-small perturbations. The Theorem \ref{thm:Rei} restricts to this particular case as follows. The regions described in the Theorem behave as follows for a generic foliation:
\begin{itemize}
\item The regions described in the points $(1)$, $(2)$ and $(3)$ are finite in number,
\item the other regions do not exist.
\end{itemize}

There is a local linearization Theorem for the neighborhood of a closed leaf with non-degenerate linear monodromy.
\begin{proposition}
Let $U \simeq S^1 \times (-\epsilon, \epsilon)$ a neighborhood of a non-degenerate closed leaf $\SF  \simeq S^1 \times \{ 0 \}$ of a foliation $\SL$, then there is a smaller neighborhood $\SF \subset V \subseteq U$ and a diffeomorphism $\phi: V \to S^1 \times (-\delta, \delta)$, for some $\delta>0$, such that,
\begin{enumerate}
\item $\phi^{-1}(S^1 \times \{ 0 \} )= \SF$,
\item fixing coordinates $(\theta,r) \in S^1 \times (-\delta, \delta)$, we have $\phi^*(\ker \{dr + \lambda \cdot r \cdot d\theta \})= \SL$ for some $\lambda\neq 0$.
\end{enumerate}
\end{proposition}
\begin{proof}
By shrinking $U$ if necessary, we may assume that $\SL$ can be expressed as $\ker \alpha= \ker  \{f\cdot dr +g\cdot d\theta \}$ for some $f,g$ functions on $S^1 \times (-\epsilon, \epsilon)$ with $f>0$. So we assume that the foliation is expressed by the kernel of the form (after suitable re-scaling)
\begin{equation}
\alpha= dr + Hd\theta, \label{eq:non-deg}
\end{equation}
with $H:S^1 \times (-\epsilon, \epsilon) \to \R$, satisfying that $H(\theta, 0)=0$. Moreover, there exists a change of coordinates $\hat{r}= \hat{r}(r)$ such that $\hat{r}(0)=0$ satisfying that the foliation is expressed as the kernel of a new form
\begin{equation}
\hat{\alpha}= d\hat{r} + \hat{H}d\theta, \label{eq:non-deg}
\end{equation}
where the function $\hat{H}$ satisfies the same properties as the previous smooth function $H$. The only difference being that the Poincar\'e's return map associated to the transverse segment $m: (- \hat{r}_0,\hat{r}_0) \to (-\hat{r}_1,\hat{r}_1)$ defined for some small $\hat{r}_0>0$ and $\hat{r}_1>0$ satisfies that is purely linear $m(\hat{r})= c\cdot \hat{r}$, for some $c>0$ and $c\neq 1$. This is proved by using the non-degeneracy condition and the classical Fatou's Lemma on the linearization of contracting germs of diffeomorphisms in the real line \cite{AR95}.

Thus we may assume, without loss of generality, that the Poincar\'e's return map is linear in the initial coordinates provided by the equation (\ref{eq:non-deg}). We can restrict ourselves to the domain $[0,1) \times (-\delta, \delta) \subset S^1 \times (-\delta, \delta)$ with coordinates $(\theta,r)$. We can easily change coordinates in that domain to $(\theta,r)=\phi(\theta,R)= (\theta,r(\theta,R))$ in such a way that $r(0,R)=R$ satisfying that $\ker \phi^* \alpha= dR$. This change of coordinates is given by the solution of the differential equation
$$ \frac{\partial r}{\partial \theta} +H(r,\theta)=0, $$
with initial value $r(R,0)=R$. This provides a unique diffeomorphism $\phi$ that is well defined over $[0,1) \times (-\epsilon, \epsilon)$ for some small $\epsilon>0$. We change the coordinates using the diffeomorphism $(R, \theta)=\Phi(\rho, \theta)= (R(\rho, \theta), \theta)$ with $R(\rho, \theta)=\rho e^{\lambda \theta}$. This implies that $\alpha_{\lambda}=\ker \Phi^* (dR)= d\rho +\lambda \rho d\theta$ that is well-defined in $(\theta, \rho)\in [0,1) \times (-\epsilon', \epsilon')$ for some $\epsilon'>0$. By adjusting $\lambda>0$, we can make the parallel transport along the foliation  coincide with the Poincar\'e's return map. If this is the case  the chart $\Phi \circ \phi$ smoothly extends to a diffeomorphism in $S^1 \times (-\epsilon', \epsilon')$. This completes the proof.
\end{proof}

From now on, the previously constructed neighborhood $V$ of a periodic orbit will be called a non-degenerate annulus. A neighborhood of the zero section of the cotangent bundle $T^*S^1$ with the vertical Lagrangian foliation will be called a cotangent annulus. Recall that its Geometric Quantization has been computed in Lemma \ref{lem:tira}.

\subsection{Geometric Quantization of the torus: The computation}
We can summarize the discussion above in the following,
\begin{corollary} \label{coro:picture}
Let $\SL$ be a generic foliation of the torus, with $N$, $N >0$, closed leaves. Then there exists a finite covering by open sets $\{ V_j \}_{j=1}^{j= N+1}$ such that:
\begin{enumerate}
 \item $V_j$ is diffeomorphic to $S^1 \times (0,1)$,
 \item $V_j \bigcap V_k= \emptyset$, if $j-k \neq \pm 1 \mod N$,
\item $V_j \bigcap V_{j+1}$ is diffeomorphic to a cotangent annulus.
\item If $j \leq N$, $V_j$ is diffeomorphic to a a non-degenerate annulus.
\item $V_{N+1}$ is a cotangent annulus.
\end{enumerate}
\end{corollary}
This result is immediate from all the previous considerations (see Figure \ref{fig:foli}).
\begin{figure}[ht]
\includegraphics[scale=0.5]{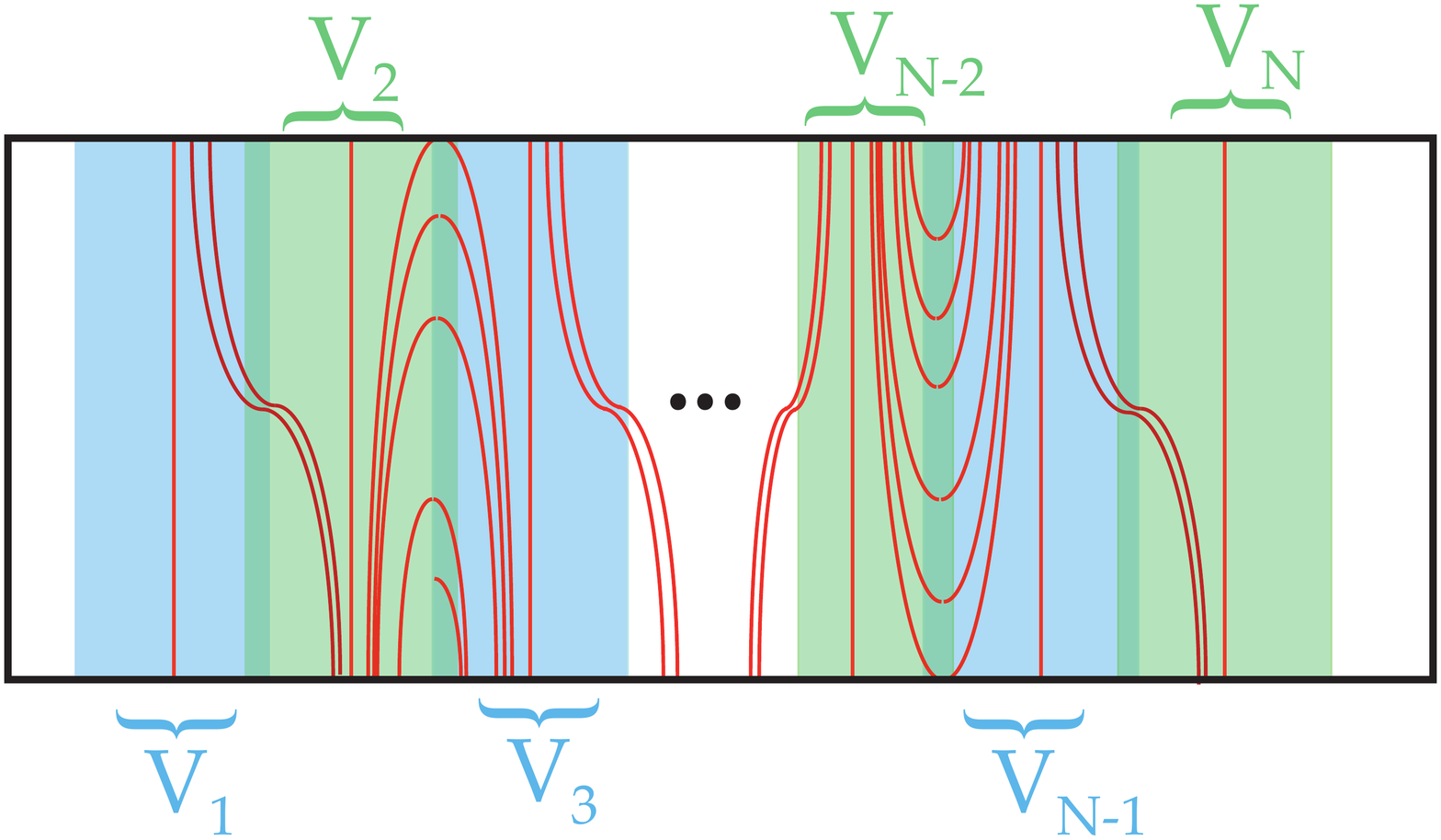}
\caption{Covering by open sets $\{ U_i \}$ of a torus with a generic foliation. The red curves are the leaves.}\label{fig:foli}
\end{figure}

In order to compute the Geometric Quantization and the foliated cohomology of a non-degenerate foliation of the torus, we just need to apply Mayer-Vietoris and the following,
\begin{lemma} \label{lem:generic}
Let $V=S^1 \times (-1,1)$ be a non-degenerate annulus with standard foliation provided by
$$\ker \alpha= \ker \{ dr +\lambda\cdot  r\cdot d\theta \}.$$
Then for any symplectic form (including the limit  foliated cohomology case with $\omega=0$),
\begin{itemize}
\item If the closed leaf is Bohr-Sommerfeld, then $H^0(V, \SJ)=\C$, otherwise $H^0(V, \SJ)=0$.
\item $H^1(V, \SJ)= \C$.
\end{itemize}
\end{lemma}

\begin{remark}
Recall that in  foliated cohomology case, any closed leaf is Bohr-Sommerfeld. Also observe that the constant $\lambda$ may be always assumed to be negative (possibly after composition with an orientation reversing diffeomorphism ($\theta \to -\theta$)).

\end{remark}
\begin{proof}
If the leaf is Bohr-Sommerfeld then the sheaf admits a parallel section over it $s_B:S^1 \to L$. Fix a non-compact leaf determined by the unique point $p_{\theta}=(\theta, 1)$ in which it intersects the boundary. The bundle $L$ is topologically trivial and it admits a connection $\nabla = d + \alpha$. Since $\SF$ is Bohr-Sommerfeld we may assume that $\alpha_{| \SF}=0$ and thus
$$\alpha = r \beta, $$
for some $1$-form $\beta$. Fix a value $s(p_{\theta}) \in L$, by parallel transport it uniquely extends to a parallel section over the whole leaf. Therefore if we fix any section $s(p_{\theta}):S^1 \times \{ 1 \} \to L$, it extends to a unique parallel section $\hat{s}(p_{\theta}):S^1 \times (0,1] \to L$. We want to check how it extends to the boundary $S^1\times \{ 0\}$. Fix a value $\theta_0 \in S^1$. We want to compute the following limit
\begin{equation}
\lim_{r\to 0} \hat{s}(\theta_0, r). \label{eq:limit}
\end{equation}
The invariance of the section along the leaves of the section allows us to reduce the existence of the
$$ \lim_{(\theta,r) \to (\theta_0,0)} \hat{s}(\theta,r),$$
to the proposed limit along the curve $(\theta_0,r)$.

Indeed, the computation of the previous limit (\ref{eq:limit}) is ensured in two steps. First, fix a point $(\theta_0,r_0)$, we will compute
the following limit
\begin{equation}
\lim_{n\to + \infty} \hat{s}(\theta_0, m^n(r_0)), \label{eq:lim2}
\end{equation}
where $m^n$ is the $n$-th iterate of the Poincar\'e's return map. The key point is to check the following inequality
$$\lim_{n\to \infty} ||\hat{s}(\theta_0, m^{n+1}(r_0))- \hat{s}(\theta_0, m^n(r_0))|| \leq C e^{-n\cdot \lambda} $$
that holds in the case in which $S^1 \times \{ 0 \}$ is Bohr-Sommerfeld. In any other case, the limit does not exist and we get that $H^0(V, \SJ)=0$. The previous inequality implies the existence of the limit (\ref{eq:lim2}).

Secondly we study the section $\hat{s}(\theta_0, r)$ restricted to the interval $\{ \theta_0  \} \times[r_0, m^1(r_0)]$. It is completely determined by the value at $r_0$ and by the condition that the limit (\ref{eq:lim2}) has to be independent of the choice of $r\in  [r_0, m^1(r_0)]$. This shows the existence of the limit (\ref{eq:limit}). Therefore each element of $H^0(V, \SJ)$ can be recovered out of a parallel section at $S^1 \times \{ 0 \}$. More precisely, we have recovered a unique parallel section over $S^1 \times [0,1]$, but an analogous argument recovers a section over $S^1 \times [-1,0]$ and the two glue together to provide a global smooth section. This shows that $H^0(V,\SJ)=\C$ if the closed leaf is Bohr-Sommerfeld and $H^0(V,\SJ)=0$ otherwise.

To compute the group $H^1(V, \SJ)$, we start by an element of $\Omega^1(V,\SJ)$ and we study under which conditions is exact. The restriction to any non-compact leaf is exact because of the standard (parametric) Poincar\'e lemma applies to the real line. In the case of the closed leaf the condition to be exact amounts to be in the kernel of an linear operator $\Omega(V,\SJ) \to \C$, i.e. the integral along the leaf $S^1 \times \{ 0 \}$. Let us detail that operator. Fix an element $\gamma \in \Omega^1(V,\SJ)$. We want to construct a section $s:S^1 \times \{ 0\}  \to L$. We fix an arbitrary (non-zero) value $v$ at $L_{(0,0)}$. There is a unique section in $[0,1) \times \{ 0 \}$ such that:
\begin{itemize}
\item $s(0,0)= v$,
\item $d_{\SF} s(\theta,0)  = \gamma(\theta,0)$, $\forall  \theta \in (0,1)$.
\end{itemize}
Consider
\begin{eqnarray*}
P: \Omega^1(V, \SJ) & \to &L_{(0,0)}\simeq \C \\
\gamma & \to & s(1,0)-s(0,0)
\end{eqnarray*}
Observe that $\gamma$ is exact if and only if $P(\gamma)=0$. Therefore $H^1(V, \SJ)= Im P = \C.$ This concludes the proof.

\end{proof}

Lemma \ref{lem:tira} and \ref{lem:generic} compute the Geometric Quantization provided by all the types of open sets appearing in  Corollary \ref{coro:picture}. Therefore, the computation of the Geometric Quantization becomes a simple task,
\begin{corollary}
Let $\SF$ be a non-degenerate regular foliation over the torus with $N>0$ closed leaves. Assume that $0 \leq b\leq N$ of them are Bohr-Sommerfeld. Then,
\begin{itemize}
\item $H^0(\mathbb T^2,  \SJ)=\C$ if $b=N$ and a parallel transport condition (to be described in the proof) is fulfilled. $H^0(\mathbb T^2, \SJ)=0$ otherwise.
\item $H^1(\mathbb T^2, \SJ)= \bigoplus_{i=1}^N (C^{\infty}(S^1, \C)/(\C)^{b(i)}) \bigoplus \C^N$, where $b(i)=1$ if the $i$-th closed leaf is Bohr-Sommerfeld; $b(i)=0$ otherwise.
\item For the foliated cohomology case $H^0(\mathbb T^2,  \SJ)=\C$ and 

\noindent $H^1(\mathbb T^2, \SJ)= \bigoplus_{i=1}^N (C^{\infty}(S^1, \C)/\C) \bigoplus \C^N$.
\end{itemize}
\end{corollary}
\begin{proof}

In order to compute  $H^0(\mathbb T^2,\SJ)$ we fix a section $s_0\in \Omega^0(\mathbb T^2,\SJ)$. When restricted to any closed leaf $\gamma_i$ $(i=1\dots, n)$ it induces a  section $s_i:\gamma_i\longrightarrow \SJ$. Recall that $s_{i+1}$ is determined from $s_i$ therefore if one of them vanishes $s=0$. So if any closed leaf is not Bohr-Sommerfeld then $H^0(\mathbb T^2,\SJ)=0$. Otherwise, $s_i$ determines $s_{i+1}$ and at the end there is a parallel transport equation to be satisfied by $s_1$. If it is satisfied $H^0(\mathbb T^2,\SJ)=\mathbb C$, otherwise $H^0(\mathbb T^2,\SJ)=0$.
In order to compute $H^1(\mathbb T^2,\SJ)$, we recall the Mayer-Vietoris Lemma (Lemma \ref{coro:Mayer}). It states that for two open domains $U$ and $V$ we have the following exact sequence:
\begin{equation*}
\xymatrix{
0 \ar[r] & H^0(M, \SJ) \ar[r] & H^0(U, \SJ)\oplus H^0(V, \SJ) \ar[r] & H^0(U \cap V, \SJ) \ar[dll] \\
& H^1(M, \SJ) \ar[r] & H^1(U, \SJ)\oplus H^1(V, \SJ) \ar[r] & H^1(U \cap V, \SJ)
}
\end{equation*}
Away from the open sets defined in Corollary \ref{coro:picture}, we define the sequence of open sets $U_j= \bigcup_{k=1}^j V_k$. It is clear that $U_{N+1}=\mathbb T^2$. It is simple to check that
$$H^1(U_{j+1}, \SJ) = \bigoplus_{i=1}^j (C^{\infty}(S^1, \C)/(\C)^{b(i+1)}) \bigoplus \C^{j+1}, j=1,\ldots, n-1,$$
by sequentially applying the Mayer-Vietoris sequence to the sets $U_{j+1}=U_j \bigcup V_{j+1}$. All the instances follow the same pattern and a recursive argument applies. The exception is the last one, being $\mathbb T^2=U_{N+1}=U_N \bigcup V_{N+1}$, in which the algebraic computation changes since $V_{N+1}$ is a cotangent annulus.

In the foliated cohomology case, the same proof holds since all the leaves are Bohr-Sommerfeld.
\end{proof}

 Notice that the condition on closed leaves not to be Bohr-Sommerfeld is a  generic one.

 Thus the Quantization space depends only on the topology of the foliation. Concretely, just in the number of closed orbits. In the general case, it depends on the symplectic geometry of the foliation, i.e. the number of Bohr-Sommerfeld leaves is invariant under symplectic diffeomorphisms. In a future work \cite{MP13} we plan to  use this idea to define the Quantization of a general Hamiltonian on a torus, since the Quantization spaces remain unchanged by the flow of the Hamiltonian and construct explicit isomorphisms of the Quantization spaces.

\end{document}